\documentclass[oneside,11pt]{amsart}

\usepackage{amsmath,amssymb,amsthm, mathrsfs, mathtools}
\usepackage{graphicx}
\usepackage{amscd,mathtools}
\usepackage{enumitem} \setlist{nosep} % Change enumeration labelling
\usepackage[utf8]{inputenc}
\usepackage[english]{babel}
\usepackage{color}
\usepackage[pagebackref,breaklinks,hidelinks,unicode]{hyperref}
\usepackage{float}
\usepackage{comment}
\usepackage{mathtools}
\usepackage{csquotes}
\usepackage[font=small,justification=centering]{caption}
\usepackage{subcaption}
\usepackage[dvipsnames]{xcolor}
\usepackage{mathabx} % Includes \sqsubset etc.
\usepackage{xfrac} % Nice quotients
\usepackage[normalem]{ulem} % Contains \sout
\usepackage{dirtytalk}

\usepackage[margin=2.98cm]{geometry}

\usepackage{tikz}
\usepackage{tikz-cd}
\usetikzlibrary{automata}
\usetikzlibrary{calc,decorations.pathreplacing}
\usetikzlibrary{arrows}
\usetikzlibrary{decorations.pathreplacing}
\usetikzlibrary{intersections}

\makeatletter
    \newtheorem*{rep@theorem}{\rep@title}
    \newcommand{\newreptheorem}[2]{%
    \newenvironment{rep#1}[1]{%
    \def\rep@title{#2 \ref{##1}}%
    \begin{rep@theorem}}%
    {\end{rep@theorem}}}
\makeatother

\newtheorem{theorem}{Theorem}[section]
\newtheorem{proposition}[theorem]{Proposition}
\newreptheorem{proposition}{Proposition}
\newtheorem{corollary}[theorem]{Corollary}
\newtheorem{lemma}[theorem]{Lemma}

\newtheorem{Set up}[theorem]{Set-up}

\newtheorem{introthm}{Theorem}

\theoremstyle{definition}

\newtheorem{definition}[theorem]{Definition}

\newtheorem{notation}[theorem]{Notation}
\newtheorem{remark}[theorem]{Remark}

\newtheorem{rankr}[theorem]{Cubical Rank Rigidity}
\newtheorem{rankrh}[theorem]{Hierarchical Rank Rigidity}

\newtheorem*{answer*}{Answer}
\newtheorem*{application*}{Application}

\DeclarePairedDelimiterX{\Norm}[1]{\lVert}{\rVert}{#1}
\theoremstyle{definition}

  \newcommand{\calC}{\mathcal{C}}

  \newcommand{\calG}{\mathcal{G}}

  \newcommand{\calU}{\mathcal{U}}
  \newcommand{\calV}{\mathcal{V}}
  \newcommand{\calW}{\mathcal{W}}

\renewcommand*{\backrefalt}[4]{\ifcase #1 (Not cited).\or (Cited p.~#2).\else (Cited pp.~#2).\fi} % Cited on...

\newcounter{shcount}
 
\newcounter{enumlabelcount}
\makeatletter\newcommand\enumlabel[1][]{\item[#1]
    \refstepcounter{enumlabelcount}\def\@currentlabel{#1}}\makeatother

\definecolor{harrycomment}{rgb}{0.6,0,0.4}

\DeclareMathOperator{\Aut}{Aut}

\DeclareMathOperator{\dist}{\mathsf{d}}
\DeclareMathOperator{\Dist}{\mathsf{D}}
\DeclareMathOperator{\diam}{diam}
\DeclareMathOperator{\vbig}{Big}
\DeclareMathOperator{\vvbig}{Vbig}

\newcommand{\hull}{\mathrm{hull}}

\DeclareMathOperator{\isom}{Isom}

\newcommand*{\X}{X_{\Dist}}

\newcommand*{\cal}{\mathcal}
\renewcommand{\bf}{\mathbf}

\newcommand*{\nest}{\sqsubset}

\newcounter{claimcount}

\newenvironment{claim*}[1]{\par\vspace{2mm}\noindent
    \underline{Claim:}\hspace{2mm}#1}{}
\newenvironment{claimproof}[1]{\par\vspace{2mm}\noindent\underline{Proof:}\hspace{2mm}#1}
    {\leavevmode\unskip\penalty9999\hbox{}\nobreak\hfill\quad\hbox{$\diamondsuit$}\vspace{2mm}}

\title[Effective rank rigidity]{Effective flipping, skewering and rank rigidity for cubulated groups with factor systems} 
\hypersetup{pdftitle = Hyperbolic models for CAT(0) spaces}
\author{Abdul Zalloum}
\begin{document}

\maketitle

\begin{abstract} Relying on work of Caprace and Sageev \cite{capracesageev:rank}, we provide an effective form of rank rigidity in the context of groups virtually acting freely cocompactly on a CAT(0) cube complex with a factor system. We accomplish this by exhibiting a special pair of hyperplanes that can be skewered uniformly quickly. Furthermore, for virtually special compact groups, we prove an effective omnibus theorem and provide a trichotomy implying a strong form of an effective Tits alternative. More generally, we provide a recipe for producing short Morse elements generating free stable subgroups in any virtually torsion-free hierarchically hyperbolic group (HHG) which recovers Mangahas' work \cite{MangahasRecipie} and provides an effective rank-rigidity dichotomy in the context of HHGs via Durham-Hagen-Sisto \cite{Durham2017-ce}. Part of our analysis involves showing that Caprace-Sageev's cubical tools of flipping and skewering can be applied to any HHG using the notion of a \emph{curtain} recently introduced by Petyt, Spriano and the author in \cite{PSZCAT}, \cite{Zalloum23_injectivity}. Indeed, our route to producing a short Morse element proceeds by exhibiting a special pair of curtains in the underlying HHG that can be skewered uniformly quickly.
\end{abstract}

\section{Introduction}

A groundbreaking work of Caprace and Sageev \cite{capracesageev:rank} shows that groups acting properly cocompactly on CAT(0) cube complexes admit the following dichotomy, which establishes the long-standing CAT(0) Rank Rigidity Conjecture in the context of CAT(0) cube complexes.

\begin{rankr}[{\cite{capracesageev:rank}}] Let $G$ be a group acting properly cocompactly on a CAT(0) cube complex $X$. Precisely one of the following occurs:

\begin{enumerate}  
\item $X=X_1 \times X_2$ where each $X_i$ is an unbounded CAT(0) cube complex, or
\item The group $G$ contains a rank-one element.
  
\end{enumerate}
\end{rankr}

A natural question to ask (as asked by Hagen in Question 6.4 of \cite{HagenFacing2022}) is whether there is an \emph{effective form} of rank rigidity. Namely, for a cubical group $G$, it's quite desirable to know that if rank-one elements exist, then it's possible to find one uniformly quickly independently of a finite generating set for $G.$ Establishing this would imply that the question of whether the cubical group in question is quasi-isometric to a product can be determined locally, by merely checking a ball of a uniformly finite radius in $G$ regardless of the chosen finite generating set. We prove such a statement for a broad class of CAT(0) cube complexes including virtually special compact groups \cite{Haglund2007},\cite{HHS1}, the Burger-Mozes groups \cite{Burger1997}, \cite{HHS1}, the BMW-groups in the sense of Caprace \cite{Caprace2019}, \cite{HHS1}, the mapping class group of genus-two handlebody \cite{miller2020stable}, some Artin groups \cite{Haettel2020} as well as cubulated graph products \cite{BR_combo} and certain cubulated quotients of the above objects \cite{BHMScombo}.

\begin{introthm}\label{thmi:RR_CCC}$(\mathrm{Effective \,rank \,rigidity})$ Each CAT(0) cube complex with a factor system $X$ determines a constant $m=m(X)$ such that for any group $G$ acting freely cocompactly on $X$ and any finite generating set $T$ for $G,$ precisely one of the following occurs:

\begin{enumerate}
    \item $X=X_1 \times X_2$ where each $X_i$ is an unbounded CAT(0) cube complex, or
    \item $\mathrm{Cay}(G,T)$ contains a rank-one element in the ball of radius $m$ centered at the identity.

\end{enumerate}
\end{introthm}

We note that the main conclusion of the above theorem does not require the $G$-action to be free, it suffices that $G$ acts virtually freely (and cocompactly) on $X$. However, the constant $m$ in this case will depend on $X$ \emph{and} on the index of a finite index torsion-free subgroup of $G$. In particular, $m$ will be a function of both $X$ and $G$ and not just $X.$ The same comment applies to Theorem \ref{thmi:short_Stable_CCC} and Theorem \ref{thmi:effective_skewer}.

If item (2) of Theorem \ref{thmi:RR_CCC} occurs, we can deduce more. Recall that each CAT(0) space $X$ admits a $\delta$-hyperbolic space called the \emph{curtain model} where all rank-one elements of $\isom(X)$ act loxodromically \cite{PSZCAT}.

\begin{introthm}\label{thmi:short_Stable_CCC} Each irreducible CAT(0) cube complex with a factor system $X$ determines a constant $m=m(X)$ such that whenever $G$ acts freely cocompactly on $X$ and $T$ is a finite generating set for $G,$ there exists a free all-rank-one stable subgroup $\langle g_1,g_2 \rangle<G $ with $|g_i|_T<m$ that quasi-isometrically embeds in the curtain model; unless $G$ is virtually cyclic.
    
\end{introthm}

If $X$ in Theorem \ref{thmi:short_Stable_CCC} is assumed to also be \emph{essential}, then $\langle g_1,g_2\rangle$ quasi-isometrically embeds in the \emph{contact graph} \cite{Hagen2013} as well.

Among the key tools in their proof of rank rigidity, Caprace and Sageev \cite{capracesageev:rank} show that for any half space $h^+$, it is possible to find an element $g_1 \in G$ which \emph{flips} $h^+$ (meaning that $g_1h^+ \subsetneq h^-)$ which they use to show that any half space $h^+$ can be \emph{skewered} by an element $g_2.$ Producing the rank-one element in their proof heavily relies on the aforementioned flipping and skewering. Our route to proving Theorem \ref{thmi:RR_CCC} involves some careful analysis of how quickly one can accomplish such a process. We show:

\begin{introthm}\label{thmi:effective_skewer} Each CAT(0) cube complex with a factor system $X$ determines an integer $m=m(X)$ such that for any non-virtually cyclic $G$ acting freely cocompactly on $X$ the following holds. There is a hyperplane $h$, elements $g_1,g_2 \in G$ such that for any finite generating set $T$ of $G,$ we have the following:

\begin{enumerate}

\item $|g_1|_T,|g_2|_T<m$,
    \item (effective flipping) $g_1h^+ \subsetneq h^-$, \,$g_2h^- \subsetneq h^+$, and
 
 \item (effective skewering) $(g_2g_1) h^+ \subsetneq h^+.$ 
\end{enumerate}
Further, the element $g_2g_1$ is rank-one.

\end{introthm}

The free action assumption in the previous Theorem \ref{thmi:RR_CCC}, \ref{thmi:short_Stable_CCC} and \ref{thmi:effective_skewer} is merely required to assure the existence of a short infinite order element (which again, can also be assured under the weaker assumption that $G$ acts virtually freely, but at the cost of $m$ having a dependency on $G$), namely, we prove the following more general statement.

\begin{introthm}\label{thmi:general_CCC} Each irreducible CAT(0) cube complex with a factor system $X$ determines a constant $m=m(X)$ such that for any group $G$ acting properly cocompactly on $X$, if $a \in G$ has infinite order, then there exists an element $g \in G$ such that $a^m(ga^mg^{-1})$ is rank-one and $|g|_T<m$ for any finite set $T$ generating $G.$
\end{introthm}
It's worth noting that the element $a^m(ga^mg^{-1})$ is shown to be rank-one by finding a special hyperplane $h$ where $ga^mg^{-1}$ flips $h^+$ and $a^m$ flips $h^-$ leading the element $a^m(ga^mg^{-1})$ to skewer $h^+$ (i.e., 
$a^m(ga^mg^{-1})h^+ \subsetneq h^+$).

Most known examples of proper cocompact CAT(0) cube complexes do admit a factor systems, in fact, and until very recently, it was conjectured that every cubical group admits a factor system \cite{HagenSusse2020}. However, fresh work of Shepard in \cite{Shepard2022} provides a remarkable counter example to the conjecture. The class of cubulated groups with factor systems lives in the much broader class of \emph{hierarchically hyperbolic groups} (HHGs). Aside from cubulated groups with factor systems, the class of HHGs includes mapping class groups \cite{MM99,MM00, HHS1}, extra-large type Artin groups \cite{HMS21}, the genus two handlebody group \cite{miller2020stable}, surface group extensions of lattice Veech groups \cite{DDLS1,DDLS2} and multicurve stabilizers \cite{russell2021extensions}, and the fundamental groups of 3–manifolds without Nil or Sol components \cite{HHS2}, as well as various combinations and quotients of these objects \cite{BHMScombo, BR_combo}.

The original motivation for introducing HHGs was the observation that the very rich hierarchical structure enjoyed by mapping class groups \cite{MM00} \cite{MM99} is also enjoyed by various cubulated groups \cite{HHS1}, \cite{HHS2}. This observation led to the introduction and formalism of the class of HHGs. Unexpectedly, studying both mapping class groups and cube complexes from the more general HHG point of view led Behrstock, Hagen and Sisto \cite{HHS_quasi} to realize that not only can one export mapping class group techniques to study CAT(0) cube complexes, but one can also import CAT(0) cube complex techniques to study mapping class groups and this observation has had a great influence on the area. First, it was the key to establishing Farb’s quasi-flats conjecture \cite{HHS_quasi} and more recently, it was heavily used to establish semi-hyperbolicity of mapping class groups by Haettel-Hoda-Petyt \cite{HHP} and Durham-Minsky-Sisto \cite{DMS20}. It has also led to the recent surprising result of Petyt that each mapping class group is quasi-isometric to a CAT(0) cube complex \cite{Petyt21}.

This work encapsulates both philosophies: we apply the well-established mapping class group large link tool \cite{MM00}, \cite{MM99} (if the projected distance to some subsurface is $K$,
then the number of maximal subsurfaces thereof with large projection
is linearly bounded in terms of $K$) to the context of cubulated groups with factor systems, and on the other hand, we import the standard flipping and skewering cubical tools to the class of HHGs via the notion of a \emph{curtain}. Combining both such tools amounts to producing the desired rank-one element uniformly quickly.

The class of HHGs enjoys a similar rank-rigidity dichotomy as established in \cite{Durham2017-ce} by Durham-Hagen-Sisto (a different and easier proof was given later in \cite{PetytSpriano20} by Petyt and Spriano).

\begin{rankrh}[{\cite{DMS20}, \cite{PetytSpriano20}}] Let $G$ be an HHG. Precisely one of the following happens:

\begin{enumerate}
    \item $G$ is quasi-isometric to a product of two unbounded HHSes, or
    \item $G$ contains a Morse element.
\end{enumerate}
\end{rankrh}
We shall say that an HHG is \emph{irreducible} if it is not quasi-isometric to a product of two unbounded HHSes. Recall that each irreducible HHG comes with hyperbolic space denoted $\calC S$ on which $G$ acylindrically acts. Relying on the hierarchical rank rigidity above, we prove:

\begin{introthm}\label{thmi:RR_HHS} Let $(G,\mathfrak S)$ be a virtually torsion-free HHG. There exists a constant $m=m(G)$ such that for any finite generating set $T$ of $G$, precisely one of the following holds:

\begin{enumerate}

    \item $G$ is quasi-isometric to a product of two unbounded HHSes, or
        \item There exists a Morse element $g \in G$ with $|g|_T<m$.

\end{enumerate}
Further, if $(2)$ occurs and $G$ is not a quasi-line, then $G$ contains a free all-Morse stable subgroup $\langle g_1,g_2 \rangle$ with $|g_i|_T<m$ that quasi-isometrically embeds in $\calC S$.
\end{introthm}

It is worth noting that most known HHGs are virtually torsion free and until very recently, it was unknown whether non-examples exist. However, in recent work \cite{HughesHHs2022}, Hughes  provided the first non-example of this kind. Also, we remark that Theorem \ref{thmi:RR_HHS} above recovers Mangahas' main theorem \cite{MangahasRecipie} in the mapping class group case where the author provides a recipe for producing pseudo-Anosov elements uniformly quickly, depending only on the surface.

The assumption that $G$ is virtually torsion-free in Theorem \ref{thmi:RR_HHS} is only used to assure the existence of a short infinite order element. Namely, we prove the following more general statement.

\begin{introthm}\label{thmi:general_HHG}
Let $(G, \mathfrak S)$ be an irreducible HHG. There exists a constant $m=m(G)$ such that for any finite generating set $T$ and any infinite order element $a \in G$, there exists $g \in G$ such that $|g|_T<m$ and $a^m(ga^mg^{-1})$ is Morse.
\end{introthm}

Unlike the cubical case where the constant $m$ depends only on the geometry of $X$ (Theorem \ref{thmi:RR_CCC}), the dependency on $G$ in Theorem \ref{thmi:RR_HHS} and \ref{thmi:general_HHG} seems crucial; this is justified as follows: every (infinite order) cubical isometry of a CAT(0) cube complex $X$ admits a combinatorial geodesic axis \cite{haglund:isometries}, and in particular, the translation lengths of such loxodromics on $X$ are bounded below by 1 independently of the acting group, and more interestingly, independently of the properness of $G$-action on $X$. On the other hand, there are two canonical spaces an HHG properly coboundedly acts on, one of which is a given Cayley graph for $G,$ and the other is the injective metric space $X$ constructed in \cite{HHP}. In \cite{Abbott-Hagen-Petyt-Zalloum23}, Abbott, Hagen, Petyt and the author show that a uniformly proper action (in particular, a proper cobounded action) of a group $G$ on an injective metric space $X$ admits a uniform lower bound $\tau_0$ on the translation lengths of all infinite order elements in $G$. The constant $\tau_0$ resulting from our proof in \cite{Abbott-Hagen-Petyt-Zalloum23} heavily depends on the properness of the $G$-action on $X,$ and this leads to a dependence of the constant $m$ on $G$ in Theorem \ref{thmi:RR_HHS}.

Exhibiting free subgroups generated by short elements, as in Theorem \ref{thmi:RR_HHS}, is a very classical and desirable type of result in geometric group theory. This is in part because it's the standard way of establishing \emph{uniform exponential growth} for the group under consideration (see Subsection \ref{subsec:uniform_growth} for more discussion on uniform exponential growth).

There are degrees of ``how free" a subgroup of a given group can be. Namely, to say that $H=\langle g_1, g_2 \rangle $ is free provides information about how elements in $H$ interact with one another but no information about how they interact with the ambient group. For instance, if $\mathbb{F}_2 \times \mathbb{Z}<G,$ then, although $\mathbb{F}_2$ is free, there is an infinite order element that commutes with every element in $\mathbb{F}_2$ and in particular, the free group $\mathbb{F}_2$ does not interact with the rest of the group in a free or hyperbolic-like fashion. A much stronger and more desirable form of a freeness or hyperbolicity a subgroup $\mathbb{F}_2<G$ might enjoy is given by \emph{stability} \cite{durhamtaylor:convex} (see also \cite{tran:onstrongly}). Free stable subgroups are not only free, but they also interact with the rest of the group in a free-like or hyperbolic-like manner, for instance, no infinite order element $g \in G$ can commute with an element of such a free subgroup. In the context of mapping class groups of finite type surfaces, a free subgroup $\mathbb{F}_2$ is stable precisely when it quasi-isometrically embeds in the curve graph \cite{ABD} (which happens if and only if $\mathbb{F}_2$ is convex-cocompact \cite{DurTay15} in the sense of Farb and Mosher \cite{FarbMosher02}). Analogously, for CAT(0) groups, a free subgroup is stable exactly when its orbit map quasi-isometrically embeds in the curtain model \cite{PSZCAT}.

In \cite{Abbot-Ng-Spriano}, Abbott-Ng-Spriano and Petyt-Gupta prove the remarkable theorem that a broad subclass of virtually torsion-free HHGs, including the irreducible ones, has \emph{uniform exponential growth}. In the irreducible case, they do so by producing a free subgroup uniformly quickly, however, the free subgroup they produce is generally not stable. This is not a defect of their proof but rather a reflection of the fact that their proof applies to a broader subclass of HHGs than the irreducible ones where Morse elements (and hence free stable subgroups) may not even exist.

Nonetheless, if such ``strongly free" (or, more precisely, stable) subgroups exist in $G$, then it's clearly more natural to attempt to produce these uniformly quickly as opposed to producing some unknown free subgroup quickly. In the context of mapping class groups, this result was established by Mangahas \cite{MangahasRecipie}. Since stable subgroups can only exist when Morse elements exist \cite{Durham2017-ce} (see also \cite{tran:onstrongly}, \cite{RST18}), Theorem \ref{thmi:RR_HHS} and \ref{thmi:RR_CCC} show that if such free stable subgroups exist, then they can be produced uniformly quickly.

Tits Alternative for a class of finitely generated groups $\calG$ asks the class to satisfy the dichotomy that any $G \in \calG$  must either be virtually Abelian or contain a free subgroup. An \emph{effective} Tits Alternative asks for the following form of the above dichotomy (as asked by Hagen in \cite{HagenFacing2022}, Abbott-Ng-Spriano in \cite{Abbot-Ng-Spriano} and by Sageev in private communication): For a given $G \in \cal G$, is there an integer $m=m(G)$ such that $G$ is either virtually Abelian or contains a free subgroup generated by elements $g_1,g_2$ with $|g_i|_T<m$ for any finite $T$ generating $G$? A more fine-tuned question to ask is the presence of a \emph{trichotomy} that picks out not only whether free subgroups exist, but also the ``freeness-level" of the provided free subgroup. More precisely, as discussed above, it is very natural to question if the free subgroup provided by Tits alternative interacts well with the rest of the group, i.e., whether it is stable. In the context of compact special groups, we provide such a trichotomy.

\begin{introthm}$\mathrm{(Strongly \,effective \, Tits \, alternative)}$\label{thmi:strongly_effective} Let $G$ be a virtually special compact group. There exists some $m=m(G)$ such that for any finite generating set $T$ for $G$, precisely one of the following happens:

\begin{enumerate}
     
\item $G$ is virtually Abelian,
\item $G$ contains a free all-rank-one stable subgroup $\langle g_1,g_2 \rangle$ with $|g_i|_T<m,$ or

\item $G$ contains a free subgroup $\langle g_1,g_2 \rangle$ with $|g_i|_T<m$ and has empty Morse boundary.
\end{enumerate}

\end{introthm}

A particular consequence of the above theorem is that a special compact group $G$ is either virtually Abelian or contains a free subgroup generated by a uniformly short pair $g_1,g_2 \in G$ which can be seen as an effective Tits alternative. Although they don't explicitly point this out, this particular consequence of Theorem \ref{thmi:strongly_effective} can also be deduced from work of Petyt and Gupta (Appendix A \cite{Abbot-Ng-Spriano}) with slightly more work. More recently \cite{Kerr2021}, Kerr established a very general theorem regarding growths of subgroups of right-angled Artin groups. A particular consequence of Kerr's theorem is that an irreducible subgroup $G$ of a right-angled Artin group $A_\Gamma$ must contain a short rank-one element for any symmetric finite generating set for $G$. It's not hard to see that if $G$ is irreducible, compact and special, then its image (up to passing to a finite index subgroup; via Haglund and Wise \cite{Haglund2007}) in $A_\Gamma$ must contain elements acting as rank-one on the universal cover of the Salvetti complex $\Tilde{X_\Gamma}$. Therefore, combing the argument given in \cite{Abbot-Ng-Spriano} with work of Kerr \cite{Kerr2021} provides a complete proof for Theorem \ref{thmi:strongly_effective} under the extra assumption that the finite generating set $T$ is symmetric. Our proof of Theorem \ref{thmi:strongly_effective} does not make such an assumption neither does it use the virtual embedding of a special compact group into a right-angled Artin group. We also show:

\begin{introthm}$\mathrm{(Effective\, omnibus)}$  If $G$ virtually acts properly, cocompactly  and cospecially  on a finite product of irreducible CAT(0) cube complexes $X=\prod X_i$, then there exist  $m=m(G)$ and $g \in G$ acting as a rank-one element on each $X_i$ and loxodromically on each curtain model of $X_i$, with $|g|_T<m.$ If $X$ is essential, then $g$ acts loxodromically on each contact graph of $X_i$ as well.
    
\end{introthm}

%the free subgroup $\mathbb{F}_2$ to interact minimally with the ambient group $G$ and the existence of such a free group is only possible if the ambient group $G$ contains a Morse element (which is equivalent to rank-one for all examples of this paper). That is to say, theoerm above says `` Since the presense of such ``strong subgroups" is only possible when Morse (equivalently rank-one) elements exist, Theorem above says that if a free group in this strong sense exists, then we can produce one uniformly quickly".

%This is an information about how the pair of elements interact with one another and not about how they interact with the rest of the group. For instance, it is possible that...

%Hence, in particular, the stabilizer of each $a,$ and $b$ is virtually cyclic.

%While this very nice statement does this, their work doesn't yield any information about how the short free group interact with the rest of the group

%\textcolor{brown}{The strongest possible form of an effective Tits alternative, rank one elements are about how elements $\langle a,b \rangle $ interact with the rest of the group and not only about with one another. While these groups are free, they don't interact freely with the rest of the group (consider $F_2 \times Z$). }

%We rely on Caprace Sageev ideas of fliping skewering, work of \cite{PetytSpriano20}, Ng-Abbott-Spriano, Johanna. Another work relevant to this article is Alice.

\subsection{Uniform exponential growth}\label{subsec:uniform_growth} The \emph{growth rate} $\lambda_{G,T}$ of a group $G$ generated by $T$ measures the rate of change of the function $f_{G,T}:\mathbb{N} \rightarrow \mathbb{N}$ that assigns to each $n$ the size of the ball $B(e,n)$ in Cay$(G,T)$. A group $G$ is said to have \emph{uniform exponential growth} if there exists an $\epsilon>0$ such that the growth rate $\lambda_{G,T}>1+\epsilon$ for any finite generating set $T$ for $G.$ Numerous examples of groups with exponential growth have uniform exponential growth (for instance, see \cite{Eskin2005}, \cite{Mangahas2009}, \cite{Koubi1998CroissanceUD}, \cite{Besson2011}, \cite{Kroph-Lyman-Ng}, \cite{Bering2018}, \cite{Gupta-Jank-Ng}, \cite{Anderson2007}, \cite{KarSageev2019}, \cite{Fuj2021},\cite{Abbot-Ng-Spriano} and \cite{FujSela}). Although this is not the main motivation for this article, Theorem \ref{thmi:RR_HHS} recovers uniform exponential growth for virtually torsion-free irreducible HHGs which was proven in \cite{Abbot-Ng-Spriano}. The tools, ideas and statements established in \cite{Abbot-Ng-Spriano} by Abbott, Ng and Spriano were of a substantial significance to the present article.

\subsection{Well-orderness of growth rates} Recently, Fujiwara \cite{Fuj2021} (see also \cite{FujSela}) showed that the growth rates of a group $G$ acting non-elementarily acylindrically on a hyperbolic space $X$ is well-ordered assuming that $G$ is equationally Noetherian and contains a short loxodromic on $X$, independently of a chosen finite generating set for $G.$ Since linear groups is equationally Noetherian, we get the following.

\begin{corollary} Let $G$ be a virtually torsion-free irreducible HHG which is not a quasi-line. If $G$ is linear or equationally Noetherian, then $G$ has well-ordered growth rates.
\end{corollary}

The known examples covered by the above corollary include virtually special groups and right-angled Artin groups, however, both of these examples are covered by work of Kerr in \cite{Kerr2021} as we explain below. Up passing to a finite index subgroup, every compact special group $G=\pi_1(X)$ lives in some right-angled Artin group $A_\Gamma$ \cite{Haglund2007} and the universal cover $\Tilde{X}$ equivariantly isometrically embeds in $\Tilde{S}$ as a convex subcomplex; where $\Tilde{S}$ is the Salvetti complex corresponding to $A_\Gamma.$ While it's possible that some rank-one elements of $G$ become not rank-one for their action on  $\Tilde{S}$ (for instance, rank-one elements that stabilize a hyperplane in $\Tilde{X}$), some rank-one elements must survive. Namely, since $\Tilde{X}$ is a convex subcomplex of $\Tilde{S}$, if $\Tilde{S}$ is a product, then either $\Tilde{X}$ is a product or it lives trivially in one of the factors. Since $G$ contains rank-one elements, we know that that $\Tilde{X}$ is not a product, hence, $\Tilde{X}$ must live trivially in the product in which case Corollary 1.0.11 of Kerr applies \cite{Kerr2021}.

\subsection{Summary of the proof} Below, we will provide a detailed proof of Theorem \ref{thmi:RR_CCC} and \ref{thmi:RR_HHS} in a very special case, we have also attempted to make the proof accessible for a reader who is familiar with CAT(0) cube complexes and not necessarily HHGs. Hence, we start by providing an overview of the aforementioned. 

\vspace{5mm}

$\bf{HHS \,\, Overview}:$
An HHS $(X, \mathfrak S)$ is a geodesic space $X$ along with a collection of $E$-hyperbolic spaces $\{\calC U\}_{U  \in \mathfrak S}$ and coarsely surjective Lipshitz maps $\pi_U:X \rightarrow \calC U$ satisfying some properties. The collection of hyperbolic spaces is indexed by a set $\mathfrak S$: for each $U \in \mathfrak S$, we assign a hyperbolic space $\calC U$. Each pair of hyperbolic spaces are related by either \emph{nesting} $\nest$, \emph{orthogonality }$\perp$ or \emph{transversality }$\pitchfork$ and formally, such relations are imposed on the index set $\mathfrak S$ (whose elements are called \emph{domains}) labelling such hyperbolic spaces. Importantly (and all uniform type results rely heavily on this fact), there is an integer $N=N(\mathfrak S)$ such that the number of pairwise non-transverse domains is at most $N.$ Hence, a ``generic" pair of domains must be transverse. There is also a unique domain $S$ with $U \nest S$ for all $U \in \mathfrak S$ called the \emph{maximal} domain. When $U,V$ are nested or transverse, there are maps $\rho^U_V:\calC U \rightarrow \calC V$ that can be thought of as projections from one hyperbolic space to the other, and unless $V \nest U,$ the map $\rho^U_V$ is coarsely a point. 

\vspace{5mm}

$\bf{Tree\, of\, flats}:$ A good example to keep in mind here is the tree of flats (the universal cover of a torus wedge a circle) where the hyperbolic spaces in the HHS structure are as follows: to each flat, you assign the obvious two orthogonal lines as hyperbolic spaces and we say that two such lines are \emph{transverse} if they lie in different flats and orthogonal otherwise. Finally, you add another hyperbolic space which is simply the \emph{contact graph} of the tree of flats denoted $\calC S$ and you declare that each line nests in $\calC S.$ The maps $\rho^U_V$ are the obvious nearest point projections. Notice how a ``generic" pair of domains here are transverse as otherwise they are either lines in the same flat, or one of them is a line and the other is the contact graph (see \cite{SistoWhatIs} and \cite{russell2020convexity} for an excellent overview of the theory).

\vspace{5mm}
$\bf{An \,HHG:}$
 An HHG then is a group $G$ whose Cayley graph $X$ is an HHS and $G$ acts on $\mathfrak S$ with finitely many orbits preserving the relations $\nest, \perp, \pitchfork$ such that:

\begin{itemize}
    \item Each $g \in G$ is an isometrty $g: \calC U \rightarrow \calC gU$,
    \item  The maps $\pi_U:X \rightarrow \calC U$ are $G-$equivariant in the sense that $g\pi_U(x)=\pi_{gU}(gx),$ and

    \item  The ``projections" $\rho^U_V$ satisfy $g\rho^{U}_V=\rho^{gU}_{gV}$ for all $U,V \in \mathfrak S$ with $U \notperp V.$ Also, if $S$ is the unique maximal domain then $G.S=S.$
    
\end{itemize}

 \vspace{5mm}

$\bf{Set\,up}:$ Now, suppose that $G$ is an irreducible HHG with an HHS structure $(X, \mathfrak S)$ and maximal hyperbolic space $\calC S$. We provide a detailed proof of Theorem \ref{thmi:RR_HHS} in the simplest possible non-hyperbolic case where $G$ is torsion-free and the HHS at play has exactly 2 nesting levels and at most 2 pairwise orthogonal domains, this way, we can keep our intuition from the tree of flats present. The reader should however be aware that the general argument, especially when there are 3 or more nesting levels, is more sophisticated than a simple iteration of the argument we are about to present via a standard ``passing-up" statement. Namely, when there are say 3 nesting levels as opposed to only two, and $U$ is at the bottom of the nesting, if you apply the passing-up argument to get $U \nest W$, there is no assurance that $G.W=W$. This presents a technical difficulty that we had to also deal with (we do so in Lemma \ref{lem:effective_second_pass}). Aside from this point, the proof of the general case follows by simply inducting the forthcoming argument. All the other results of the article follow from Theorem \ref{thmi:RR_HHS} combined with some short arguments and known facts from the literature. To keep the notation simple, we may confuse domains $U$ coming from the index set $\mathfrak S$ and their associated hyperbolic spaces $\calC U$. 

\vspace{5mm}

$\bf{Proof \,(sketch):}$ Let $T$ be a finite generating set for $G$ and let $a \in T.$ The element $a$ must be loxodromic on at least one hyperbolic space $\calC U$ with stable translation length $\tau_U(a)>A,$ for some constant $A=A(\mathfrak S)$. This is proven in Theorem 1.5 of \cite{Abbott-Hagen-Petyt-Zalloum23}. If $\calC U=\calC S,$ we are extremely happy as any loxodromic on the maximal $\calC S$ must be Morse \cite{Durham2017-ce} (see also \cite{RST18}, \cite{PetytSpriano20} or \cite{ABD}), so let's assume that $U \sqsubsetneq S.$ Observe that it is not possible for every $b \in T$ to stabilize $U$ as otherwise, the product region corresponding to $U$ ($\mathbb{R}^2$ in the tree of flats example) is stabilized by the entire group violating that $X$ is irreducible. Similar to Proposition 4.2 in \cite{Abbot-Ng-Spriano}, one then considers the set of domains $\{U\} \cup \{TU\} \cup \{T^2U\}$, (in the tree of flats, this is a collection of lines since each $g \in G$ is an isometry). We now argue as follows.

\vspace{2mm}

\noindent (1) The set of domains $\calU=\{U\} \cup \{TU\} \cup \{T^2U\}$ contains a transverse pair (a pair of lines that are in distinct copies of $\mathbb{R}^2$; when the space is the tree of flats). Indeed, since $TU \neq U$, we have $\{T^2U\} \nsubseteq \{U\} \cup \{TU\}$ as otherwise $\{U\} \cup \{TU\}$ a finite collection of hyperbolic spaces which is $T$-invariant and hence $G$-invariant (in the tree of flats, this says that there exists a finite collection of lines which are invariant under $G$, but this is absurd). Hence, $\calU$ contains at least 3 domains and since the maximal number of pairwise non-transverse domains is 2, at least two of them must be transverse. Up to translating by an element from $T \cup T^2$, we may assume that the transverse domains are of the form $U,bU$ for some $b \in T \cup T^2.$ Our next goal is to use $U,bU$ to produce a large number of transverse domains of a \emph{specific type}.

\vspace{2mm}

\noindent (2) Recall since $a$ is loxodromic on $\calC U$, the element $bab^{-1}$ is loxodromic on $b\calC U=\calC bU$. In the tree of flats, the space $\calC U$ is a line $\mathbb{R}$ so $b\calC U$ is simply the translate of this line by $b.$ Further, these two lines $U,bU$ much live in distinct flats since they are transverse, call them flat 1 and flat 2 respectively.

\begin{remark} The transverse domains $U,bU$ are at distance at least 1 in the contact graph, simply because they lie in different flats. The point of the next two steps of the argument is to translate the domains $U,bU$ around to produce two transverse domains $U,gU$ which are sufficiently far in the maximal domain $\calC S$. However, in the tree of flats case, it's enough to know that the transverse domains are at distance at least one, hence, we can use the domains $U,bU$ that we already have and conclude the argument: the element $a$ acts by translation on a line in flat 1, hence it fixes a point in the contact graph; the red point in Figure \ref{fig:rotating_tree} (the graph here is locally infinite, but we omitted infinitely many edges for the sake of making the picture clearer). Similarly, the element $bab^{-1}$ fixes a point in flat 2; the blue point in Figure \ref{fig:rotating_tree}. One can then consider a geodesic connecting these two points (in the graph below, this geodesic is just an edge), pick the middle point of an edge connecting two flats, and consider the obvious two half spaces resulting from this point. Choose the half spaces that points towards the blue point. The element $a$ acts as translation on a line in flat 1, and hence it fixes the red point and rotates the space around it; in particular, after applying a high enough power of $a$, the chosen half space in the top left picture gets \emph{flipped} into its other half space as seen in the top right picture. An important point here is that the $a$-power needed to achieve such a flip is uniform (this is exactly the point of Theorem 1.5 in \cite{Abbott-Hagen-Petyt-Zalloum23} which gives a uniform lower bound on the $a$-translation on the domain $U$), let's call such a power $m$. Similarly, the element $bab^{-1}$ fixes the blue point and rotates the space around it, hence, after applying a suitable power of $bab^{-1}$, the picture in the top right of Figure \ref{fig:rotating_tree} transforms into the one in the bottom; after taking the maximum of the two powers, we may assume that such a power is also $m$. By combining the two previous steps, we see that the element $g=(ba^mb^{-1})a^m$ \emph{skewers} the half space in the top left picture properly into itself. From here, known statements from the literature imply that $g$ is rank-one, for instance, by \cite{capracesageev:rank}. The form $g=(ba^mb^{-1})a^m$ justifies the formula for the rank-one element of Theorem \ref{thmi:general_CCC}.

 \begin{figure}[ht]
   \includegraphics[width=\textwidth, trim = .001cm 7cm 2cm 5cm]{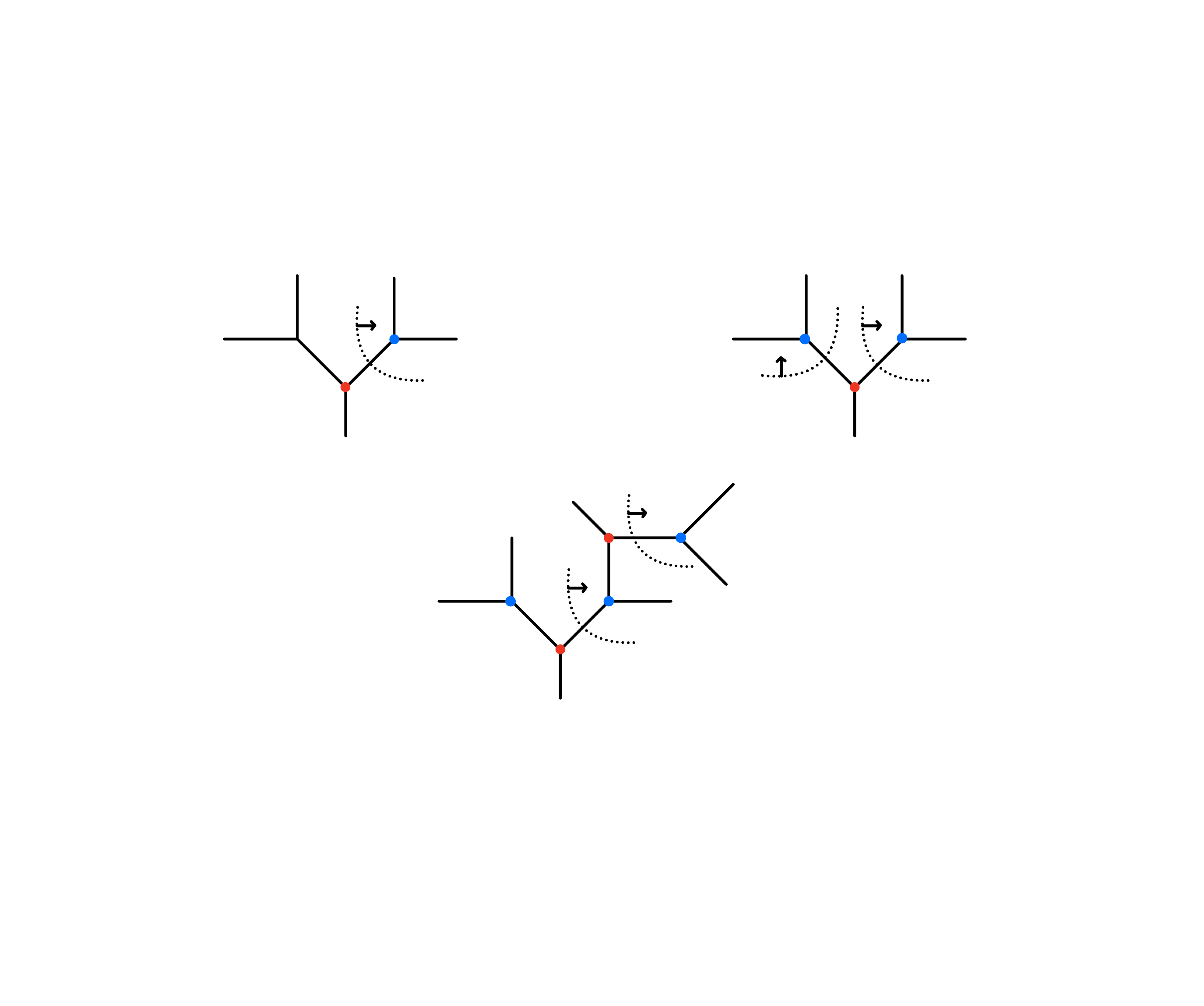}\centering
\caption{The elements $a, bab^{-1}$ act as rotations around the red and blue vertices respectively. Hence, after applying a large enough power of $a$, the half space in the top left picture gets flipped properly inside its other half space as shown in the top right picture. Applying a high enough power of $bab^{-1}$ rotates the translated half space properly inside the original half space.} \label{fig:rotating_tree}
\end{figure}
\vspace{2mm}
\end{remark}

\vspace{2mm}

\noindent (3) Continuing with the more general case, we now use $U,bU$ to produce a large number of nicely-aligned transverse domain (in a chain-like fashion so that we can implement the Behrstock inequality). It is essential that the large number of transverse domains we are about to produce are well-aligned and aren't just an arbitrary collection of transverse domains since, eventually, we want to use them to argue that we are picking out a certain direction in a higher up domain (for instance, in the tree of flats, we want the lines to be aligned in a chain-like fashion as opposed to a facing tuple fashion, to guarantee that we are moving in a particular direction and making significant distance in the contact graph). Fix a large enough constant $D$ to be determined later, depending only on the HHG and let $U_1=bU.$ Since $g_1=bab^{-1}$ is loxodromic on $\calC U_1$ with translation length $\tau_{U_1}(g_1)=\tau_{U}(a)=A=A(G),$ there exists some $m=m(D)$ such that $\dist_{U_1}(\rho^U_{U_1}, g_1^m \rho^U_{U_1})=\dist_{U_1}(\rho^U_{U_1}, \rho^{g_1^mU}_{g_1^mU_1})=\dist_{U_1}(\rho^U_{U_1}, \rho^{g_1^mU}_{U_1})>D$. On the other hand, the hyperbolic space $\calC U$ gets translated to $\calC g_1^mU$ which we denote $\calC U_2$, see Figure \ref{fig:produce_transverse_intro}. Now, we repeat the process, namely, the element $g_2^m:=g_1^ma^mg_1^{-m}$ fixes the domain $\calC U_2$ and translates the point $\rho^{U_1}_{U_2}$ to $\rho^{g_2^mU_1}_{g_2^mU_2}=\rho^{g_2^mU_1}_{U_2}$. Repeating this $n$-times (we will eventually choose a particular $n$ depending only on $\mathfrak S$) produces a collection $\calU=\{U, U_1, \cdots, U_n\}$ of transverse domains as in Figure \ref{fig:produce_transverse_intro}. By construction, we have $U_i=g_i^mU$ and each $g_i$ is a conjugate of $a.$ Also, notice that the $T$-length of $g_i$ depends only on $n,m$, and the latter will be shown to depend only on $\mathfrak S.$ In the tree of flats, what we have done is the following: the lines $U,bU$ lie in distinct copies of $\mathbb{R}^2;$ flat 1 and flat 2 respectively. The element $g_1=bab^{-1}$ acts by translation on the line $bU$ in flats 2, so we have used $g_1^m=ba^mb^{-1}$ to translate the nearest point projection $\rho^{U}_{U_1} \in bU$ far away from itself. While doing so, the line $U$ in flat 1 gets translated to a line $g_1^mU$ in a new flat, call it flat 3, and the nearest point projection of the line $g_1^mU$ to the line $bU$ in flat 2 is the $g_1^m$-translate of the nearest point projection of $U$ to $bU.$

    %In the tree of flats, all we have done is act by $bab^{-1}$ on the two lines $U,bU$. The first line, $U$, got moved to a new line $bab^{-1}U,$ the second line $bU$ is stabilized, and the distance between the projections of $U$ and $bab^{-1}U$ to $bU$ is larger than $D$ since $m$ was chosen so that $ba^mb^{-1}$ accomplishes this. See Figure.

    %In the tree of flats, the point of such a translation is to assure that the transverse lines we produced are aligned up nicely (in a chain-like fashion) which in turns implies ``picking out" a certain direction in the contact graph

\vspace{2mm}

\noindent (4) The point of translating and assuring the particular ``chain-like" alignment in the previous steps is this: Since the maps $\pi_U:X \rightarrow \calC U$ are all coarsely surjective, one can pick two points $x,y \in X$ such that $\pi_U(x)$ is far from $\rho^{U_1}_U$ and similarly, $\pi_{U_n}(y)$ is far from $\rho^{U_{n-1}}_{U_n}$, see Figure \ref{fig:produce_transverse_intro}. Hence, the Behrstock inequality assures that $\dist_{U_i}(x,y)>D$ for all $1 \leq i \leq n$. In the tree of flats $X$, if you have a collection of lines $\{\calC U_i\}$ in distinct flats where the projections of $\calC U_{i-1}, \calC U_{i+1}$ to $\calC U_i$ are $D$-far from one another, then you can find points $x,y \in X$ such that any geodesic connecting them must spend at least $D$-time ``in" each such line, in particular, this assures us that $x,y$ --and in fact the two lines $\calC U, \calC U_n$-- are far in the contact graph which is the point of the next step.

\begin{figure}[ht]
   \includegraphics[width=\textwidth, trim = 2cm 9cm 1cm 7cm]{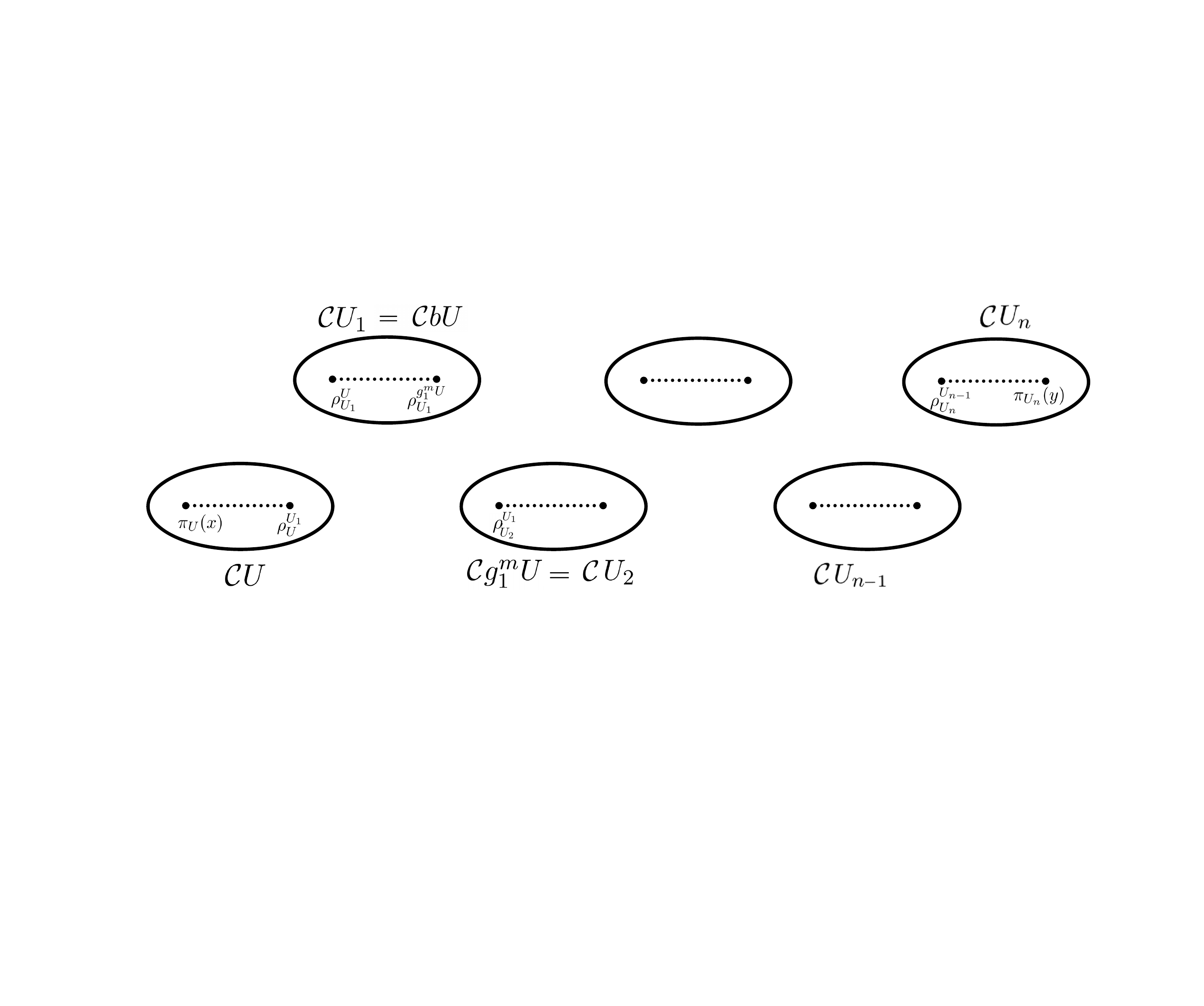}\centering
\caption{Producing many well-aligned transverse domains.} \label{fig:produce_transverse_intro}
\end{figure}
\vspace{2mm}

  \noindent (5) Since $\dist_{U_i}(x,y)>D$ for all $1 \leq i \leq n,$ the standard passing-up argument assures us that $x,y$, and in fact $U,U_n,$ are far in the unique higher up domain $\calC S$ which is the contact graph in the tree of flats case. More precisely, $\dist_S(\rho^U_S, \rho^{U_n}_S)>D$. But recall that $U_n=gU$ for some $g$ with $|g|_T$ bounded above depending only on $m,n$. For simplicity, let's just denote $gU$ by $V.$ See Figure \ref{fig:flip_skewer_intro}.

    \vspace{2mm}

\noindent  (6) This step will use the flipping and skewering tools of Caprace and Sageev \cite{capracesageev:rank}, but applied to \emph{curtains} in general HHSes. Consider a geodesic $\alpha$ connecting two points $x_1 \in \rho^U_S, x_2 \in \rho^{V}_S$ in $\calC S$ as in Figure \ref{fig:flip_skewer_intro} and pick an interval $I=[y_1,y_2] \subset \alpha$ of length $6E$ around the middle of the geodesic (as shown in Figure \ref{fig:flip_skewer_intro}), where $E$ is the uniform hyperbolicity constant for the hyperbolic spaces $\{\calC U|U \in \mathfrak S\}.$ 
    
    \begin{itemize}
        \item (curtains) Consider the nearest point projection (coarse) map $\pi_\alpha: \calC S \rightarrow \alpha$ and define $h:=\pi_\alpha^{-1}(I)$. Call such an $h$ a \emph{curtain} dual to $\alpha$ at $I.$ Similarly, define $h^+:=\pi_\alpha^{-1}((y_2, x_2])$ and $h^-:=\pi_\alpha^{-1}([x_1,y_1)])$; the green and red pieces in Figure \ref{fig:flip_skewer_intro}. 

        \item (curtains separate) Observe that $\calC S=h \cup h^+ \cup h^-$ and that $h^+, h^-$ are disjoint since $\pi_\alpha$ is $(1,E)$-Lipshitz (however, $h \cap h^+, h \cap h^-$ might be non-empty). Define $H, H^+,H^- \subset X$ to be the pre-images of $h,h^+,h^-$ under $\pi_S$. Since $\pi_S$ is an $(E,E)$-coarsely Lipshitz surjection, we have $X=H \cup H^+ \cup H^-,$ and $H^+ \cap H^-=\emptyset$  (in the cubical case, it's easy to argue that $H$ contains an actual cubical hyperplane using the thickness of the interval $I$, see Lemma \ref{lem:hyperplane_in_curtain}). 
        
        \item (HHS flipping lemma) By the bounded geodesic image theorem \cite{HHS2}, the set $\pi_{V}(H^+\cup H)$ coarsely coincides with $\rho^{U}_{V}$ which is coarsely a point.

        %In the tree of flats, our maximal domain is the contact graph $\calC S.$ By construction, the projection under $\pi_S$ of $H \cup H^+ \subset X$ to $\calC S$ is far from the point $\rho^V_S \in \calC S$ where $\rho^V_S$ is simply the projection of the line $V$ to the contact graph. In particular, since the images of the set $H^+ \cup H \subset X$ and the line $\calC V \subset X$ are far from one another in the contact graph $\calC S$, such sets must be separated by many strongly separated hyperplanes, and hence, the nearest point projection of $H^+ \cup H$ to the line $\calC V$ is a point.

        \item (flipping in $X$) Since $w=gag^{-1}$ is loxodromic on $\calC V=\calC gU$ (as $a$ is loxodromic on $\calC U$), there is a uniform power $m$, depending only on $E$ and $\tau_V(g)=\tau_U(a)=A=A(G),$ such that $\pi_V(w^m(H^+\cup H))=w^m\pi_{V}(H^+\cup H)$ is disjoint from $\pi_V((H^+\cup H))$. But the set $\pi_V((H^+\cup H))$ coarsely agrees with $\rho^U_V$. Hence, the set $H^+ \cup H$ is disjoint from $w^m(H^+\cup H),$ and as $X=H \cup H^+ \cup H^-,$ we conclude that $w^m(H^+\cup H) \subsetneq H^{-}$ (see Figure \ref{fig:flip_skewer_intro}).

        \item (flipping in $\calC S$) Since $H,H^+,H^-$ are defined to be the preimages of $h,h^+,h^-$ under $\pi_S$ and as $\pi_S$ is $G$-equivariant, we have $w^m(h^+ \cup h) \subsetneq h^-$ and $w^mh \cap h=\emptyset.$ This shows that the ``half space" $h^+$ (along with $h$ itself) got properly ``flipped" inside $h^-$, as shown in Figure \ref{fig:flip_skewer_intro}.

\begin{figure}[ht]
   \includegraphics[width=100mm, trim = .001cm 2cm 2cm 2cm]{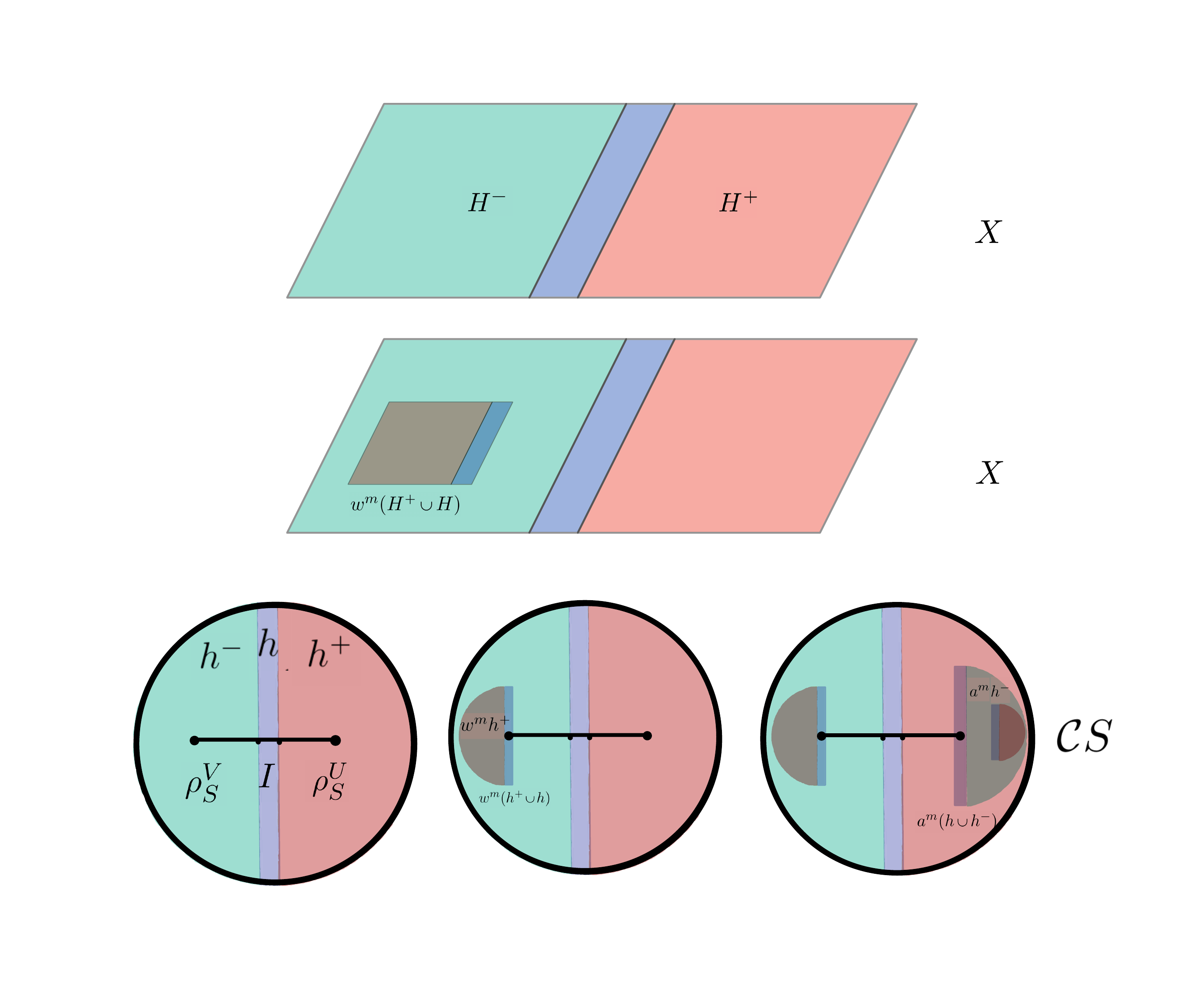}\centering
\caption{Flipping and skewering in $\calC S$.} \label{fig:flip_skewer_intro}
\end{figure}

        \item (skewering) The exact same argument shows that $a^m(h \cup h^-) \subsetneq h^+.$ In particular, $a^mw^mh^+ \subsetneq h^+$ and $a^mw^mh \cap h =\emptyset.$ It's then immediate that $a^mw^m$ is loxodromic on $\calC S$ (see Lemma \ref{lem:translation__length_via_skewering}), in fact, the stable translation length of $a^mw^m$ is at least 1. Observe that since $w=gag^{-1}$, the element $a^mw^m$ equals $a^m(ga^mg^{-1})$ which partially explains the form for the Morse element we obtain in Theorem \ref{thmi:general_HHG}. 
            \end{itemize}

\subsection{Questions} \label{subsec:further_directions} There are a few natural questions that have come up while working on this project, they are as follows.

\vspace{5mm}

\noindent\textbf{Effective rank rigidity for torsion-free cubulated groups.} The effective rank rigidity statement we provide for HHGs applies to a broad class of examples since all naturally occurring HHGs are virtually torsion-free (for instance, such a statement applies to  Artin groups of extra-large type, and more generally Artin groups of large and hyperbolic type). However, regarding questions relating to uniform exponential growth of CAT(0) cube complexes, the machinery of HHGs seems like an overkill. Namely, given the wealth of tools we have for studying CAT(0) cube complexes (see \cite{Genevois_Book} for an excellent treatment), the factor system assumption in this article is unsatisfactory; especially since the idea of what a proof should look like is very simple: given a group $G$ acting freely cocompactly on an irreducible CAT(0) cube complex $X$ (with no factor system), for a finite generating set $T$ of $G$, if all elements of $T$ stabilize a product $P$, we get that $G$ stabilizes this product which is contradiction. Hence, at least one element of $T$ must move $P$ away from itself (possibly with some overlaps), that is, $aP \neq P$ and potentially $aP \cap P \neq \emptyset.$ One then iterates the process, namely, it's again not possible that $aP \cup P$ is stabilized by every generator, so some generator moves it away from itself. If one arrives at a stage where $gP$ and $P$ are ``far enough" (in the sense that the combinatoral projection of $gP$ to $P$ is finite) for a short $g$, then it's easy to produce a rank-one element; however, it's unclear how one can arrive at such a stage. An obvious attempt would be to use the Helly property and try to argue that if for too many translates $P,gP,\cdots g^nP$ we have an unbounded combinatorial projection from $g_iP$ to $g_jP$, then too many hyperplanes intersect each other violating that $X$ is finite dimensional. However, the author was unsuccessful in making this work (it's worth pointing out that this is exactly what a factor system provides; it provides a control on how projections of such sets $g_iP$ overlap with each other).

\noindent\textbf{Producing infinite-order elements quickly.} Let $G$ act properly cocompactly on a CAT(0) cube complex (possibly with a factor system). Does there exist some $m=m(G)$ (or $m=m(X)$) such that $G$ contains an infinite order element $a$ with $|a|_T<m,$ for any finite generating set $T?$ What about HHGs? Establishing such a short element in any of the above two cases will get rid of the virtually torsion-free assumptions in the respective Theorem \ref{thmi:RR_CCC} and Theorem \ref{thmi:RR_HHS}.

\smallskip

\noindent\textbf{Strongly effective Tits Alternative.} As we discussed leading to and in Theorem \ref{thmi:strongly_effective}, a fine-tuned form  of an effective Tits Alternative asks for a trichotomy that emphasizes the convexity properties or the ``freeness level" met by the desirable free subgroup. Namely, we ask the following question. Which class of groups $\calG$ satisfy the following trichotomy? For each $G \in \calG,$ there exists an $m=m(G)$ such that for any finite generating set $T$ for $G$ precisely one of the following holds:

\begin{itemize}
    \item $G$ is virtually Abelian, 
    \item $G$ contains a free all-Morse stable subgroup $\langle g_1,g_2 \rangle $ with $|g_i|_T<m$, or

    \item $G$ contains a free subgroup $\langle g_1,g_2 \rangle $ with $|g_i|_T<m$ and has empty Morse boundary.

\end{itemize}

Recall that an element $g \in G$ is said to be \emph{Morse} if it has a quasi-axis $\alpha$ in a Cayley graph of $G$ such that each $(q,q)$-quasi-geodesic $\beta$ with end points on $\alpha$ remains in some $M=M(q,q)$-neighborhood of $\alpha$.
\smallskip

\noindent\textbf{Short rank-one elements for CAT(0) groups.} Our tools in the present article rely heavily on the notion of a \emph{curtain} introduced in \cite{PSZCAT}. The hyperbolic spaces in the hierarchichal structure of an HHS $X$ organize how such curtains interact with one another and such an organization is precisely what amounts to finding short elements that flip and skewer a curtain uniformly quickly. In the context of rank-one CAT(0) groups, \cite{PSZCAT} already provide a few pieces of the puzzle: curtains (particularly, $L$-separated curtains) are abundant and there is a single hyperbolic space where the desired short rank-one element is compelled to (loxodromically) act. Can such data be combined to obtain a short rank-one element?

\subsection*{Acknowledgements} The author is very thankful to:

\begin{itemize}

    \item Montserrat Casals-Ruiz, Sami Douba and Jacob Russell for useful discussions regarding linear groups and their connections with HHGs.

    \item Kunal Chawla and Thomas Ng for some useful feedback on an earlier draft of the paper.

    \item Alice Kerr for discussions regarding the connections between her work in \cite{Kerr2021} and this article and for patiently explaining to me many of her results in \cite{Kerr2021}.

    \item Macarena Arenas, Rylee Lyman, Michah Sageev and especially Mark Hagen for discussions related to special compact CAT(0) cube complexes.

    \item Johanna Mangahas for thoroughly explaining some of her results in \cite{MangahasRecipie} to me.

    \item Nir Lazarovitch for asking an interesting question that led to strengthening the conclusions of Theorem \ref{thmi:RR_CCC}, \ref{thmi:short_Stable_CCC} and \ref{thmi:effective_skewer}.
    
        \item Harry Petyt for explaining to me many ideas from his work with Spriano in \cite{PetytSpriano20} with Gupta in \cite{Abbot-Ng-Spriano} and for a useful discussion regarding Lemma \ref{lem:effective_second_pass}.

    \item Michah Sageev for asking a question regarding effective Tits Alternative which led to Theorem \ref{thmi:strongly_effective}.

    \item Mathew Durham for sharing with me a draft of his forthcoming paper \cite{Durham2023} within which is a useful lemma to this work.
    
    \item Kasra Rafi for teaching me many mapping class group techniques and leading me to discover that it's often very useful to think in terms of mapping class groups examples as opposed to only CAT(0) cube complexes.

    \item Thomas Ng for explaining to me some ideas from his work with Abbott and Spriano in \cite{Abbot-Ng-Spriano}.

    \item Thomas Ng, Davide Spriano, and Giulio Tiozzo for patiently listening to me explain some of this work to them.

    \item Mark Hagen for asking whether the dichotomy of Theorem A exists during his talk at the University of Toronto hyperbolic lunch on April 2020. Indeed, Hagen's question is what initiated my work on this project.

\end{itemize}

\section{Stable translation length and growth of finite index subgroups}

Let $g$ be an isometry of a metric space $X$, the \emph{stable translation length} of $g$ is defined to be $\tau_X(g)=\underset{n\rightarrow \infty}{\lim}\frac{\dist (x,g^nx)}{n}$, for some $x \in X$. It is immediate to check that for any metric space $X$, the stable translation length is independent of the chosen point $x \in X.$ An element $g$ is said to be \emph{loxodromic} on $X$ if $\tau_X(g)>0.$

For a finite generating set $T$ of a group $G$, the following lemma allows us to pass to a finite index subgroup $H$ with minimal affect to the $T$-length of elements of $G.$

\begin{lemma}[{\cite[Lemma 3.4]{Shalen1992}}]\label{lem:pass_to_finite_index} Let $G$ be a group with a finite generating set $T$ and let $H$ be a finite index subgroup with $[G:H]=d$. There exists a finite generating set $T_H$ for $H$ all of whose elements have $T$-length at most $2d-1.$
    
\end{lemma}

The following follows by combining Theorem 14 of \cite{Fujiwara2009SubgroupsGB} 
and Proposition 2.3 part 2 of \cite{Fuj08}.

\begin{theorem}[{\cite{Fuj08}, \cite{Fujiwara2009SubgroupsGB}}]\label{thm:fuj0809} If $G$ acts on an $E$-hyperbolic space $X$ with $a,b\in G$ where $a$ acts loxodromically and $a^k \neq ba^kb^{-1}$ for each $ k \neq 0,$ then there exists some $n$ such that for all $m \geq n$ we have:

\begin{enumerate}
    \item $\langle a^m, ba^mb^{-1}\rangle $ is a free, and
    \item $\langle a^m, ba^mb^{-1}\rangle $ is quasi-isometrically embedded in $X.$
\end{enumerate}
In particular, all elements of $\langle a^m, ba^mb^{-1}\rangle$ act loxodromically on $X.$ 
    
\end{theorem}

\section{CAT(0) spaces, cube complexes and coarse median spaces} For the definitions of a CAT(0) space, cube complex, coarse median spaces and some related facts, see Sections 2.4 and 2.5 of \cite{DZ22} were we give a relatively thorough treatment. Here, we will merely recall the facts most relevant to this paper. A CAT(0) cube complex is said to be \emph{finite dimensional} if there exists an integer $v$ such that if $[0,1]^k$ is an isometrically embedded cube in $X$, we have $k \leq v.$ The integer $v$ is said to be the \emph{dimension} of $X.$ A \emph{combinatorial path} is a path that lives in the $1$-skeleton of $X$. A \emph{combinatorial geodesic} connecting $x,y \in X^{0}$ is combinatorial path connecting $x,y$ of minimal length. The \emph{combinatorial distance} between $x,y$ is defined to be the length of a combinatorial geodesic connecting $x,y.$ For $x,y \in X^{0}$, we let $[x,y]$ denote the collection of all combinatorial paths connecting $x,y.$ A subset $Y\subset X$ is said to be \emph{combinatorially convex} if whenever $x,y \in Y^0$, we have $[x,y] \in Y.$

\begin{definition}(Hyperplanes, half spaces and separation) Let $X$ be a CAT(0) cube complex. A \emph{hyperplane} is a connected subspace $h \subset X$ such that for each cube $c$ of $X$, the intersection $h \cap c$ is either empty or a
midcube of $c$.  For each hyperplane $h$, the complement $X \setminus h$ has exactly two connected components $h^+, h^-$ called \emph{half-spaces}
associated to $h$ . A hyperplane $h$ is said to \emph{separate} the sets $A,B \subseteq X$ if $A \subseteq h^+$ and $B \subseteq h^-.$ A standard fact regarding CAT(0) cube complexes (for instance, see \cite{capracesageev:rank}) is that the combinatorial distance between two vertices $x,y \in X^{0}$ coincides with the number of hyperplanes separating $x,y.$ Finally, a \emph{chain} is a collection of hyperplanes $\{h_1,h_2, \cdots h_n\}$ such that each $h_i$ separates $h_{i-1}$ from $h_{i+1}.$
\end{definition}

%An essential tool (and in fact characteristic) of studying CAT(0) cube complexes is the existence of \emph{medians.}

%begin{definition}(Medians)\label{def:CCC median} We define the \emph{median} of the vertices $x, y, z \in X$ to be the unique vertex $m(x,y,z) \in X$ obtained by associating to every hyperplane of $X$ its half-space that contains the majority of the points $x, y, z$. Equivalently, $m(x,y,z)$ is the unique point that lives in the intersection of the sets of all combinatorial geodesics connecting $\{x,y\},\{x,z\}$ and $\{y,z\}.$ For two vertices $x,y \in X,$ we define the \emph{median interval} $[x,y]$ by 

%$$[x,y]=\{m(x,y,z)| z \in X^{(0)}\}.$$ Equivalently, $[x,y]$ is the union of all combinatorial geodesics connecting the vertices $x,y$.

%\end{definition}

In \cite{Bowditch13}, Bowditch introduced the notion of a \emph{coarse median space} were, roughly speaking, each finite set of points can be approximated by a CAT(0) cube complex, and for any three points $x,y,z,$ one can assign a ``coarse center" $m(x,y,z)$ satisfying certain properties (see Definition 2.28 in \cite{DZ22}). A coarse median space is said to be of dimension $v$ if the every finite set of points is approximated by a CAT(0) cube complex of dimension at most $v.$ A subset $Y$ of a coarse median space $X$ is said to be \emph{$K$-median-convex} if for any two points $x,y \in Y$ and $z \in X,$ the point $m(x,y,z)$ is within $K$ of $Y.$ Also, for two points $x,y \in X,$ one can define the \emph{median interval}, denoted $[x,y]$, as $[x,y]=\{m(x,y,z)| z \in X\}$. The primary examples of coarse median spaces are CAT(0) cube complexes and in that case, the median interval $[x,y]$ is exactly the collection of all combinatorial geodesics connecting $x,y.$ Finally, for an arbitrary subset of a coarse median space, one can form the \emph{median hull} as follows.

\begin{definition}\label{def:median_hull} (median hull)
    Let $Y$ be a subset of a coarse median space $X$, let $$J^1(Y)= \underset{x,y \in Y}{\cup} [x,y].$$ We define $J^n(Y)$ inductively by setting $J^0(Y)=Y, J^1(Y)$ as above, and $J^n(Y)=J(J^{n-1}(Y))$. By Lemma 6.1 of \cite{Bowditch13}, if $v$ is the dimension of $X$, then there exists a constant $C$, depending only on $X,$ such that $J^{v+1}(Y) \subseteq \cal N_C(J^v(Y))$. It's then immediate that, if one defines $\hull(Y)=J^v(Y),$ then $\hull(Y)=J^v(Y)$ is $C$-median convex. Further, in the special case where $X$ is a CAT(0) cube complex, $C=0$ (so $J^{v+1}(Y)=J^v(Y))$ and $\hull(Y)$ is combinatorially convex, where $v$ is the dimension of the CAT(0) cube complex $X.$
    
\end{definition}

%\begin{definition}(\label{def:combinatorial_hull}joins, hulls) Given any subset $Y \subseteq X,$ one can form the \emph{combinatorial convex hull} or simply the convex hull as follows. First, one defines $$J(Y):= \underset{x,y \in Y^{0}}{\bigcup}[x,y].$$ We define $J^n(Y)$ inductively by setting $J^0(Y)=Y$ and $J^n(Y)=J(J^{n-1}(Y)).$ If $v$ is the dimension of $X$, then $J^{v+1}(Y)=J^v(Y)$ (for instance, by Lemma 6.1 of \cite{Bowditch2019CONVEXITY}) and the \emph{convex hull} of $Y$ is defined to be $J^v(Y)$ and is often denoted by $\hull(Y)$.

%\end{definition}

Another useful fact for CAT(0) cube complexes is the existence of a \emph{combinatorial projection}.

\begin{definition}\label{def:combinatorial_projections}(Combinatorial projections) Let $Y$ be a combinatorially convex set in a CAT(0) cube complex $X,$ and let $x$ be a vertex in $X$. The \emph{combinatorial projection} of $x$ to $Y$, denoted $P_Y(x),$ is the vertex realizing the distance $d(x,Y).$ Such a vertex is unique (for instance, by Lemma 1.2.3 \cite{Genevois2015HyperbolicDG}) and it is characterized by the property that a hyperplane $h$ separates $x,Y$ if and only if it separates $x,P_Y(x).$ For such a characterization, see Lemma 13.8 in \cite{Haglund2007}. See also \cite{HHS1} and \cite{HHS2}.
\end{definition}

  \subsection{The contact graph and the curtain model} For any CAT(0) space, one can associate a $\delta$-hyperbolic space for $X$ (with $\delta$ independent of the CAT(0) space) called the \emph{curtain model} \cite{PSZCAT} and denoted $X_{\Dist}$. A fundamental property of such a hyperbolic space is that it witnesses the action of \emph{all} rank-one elements in $\isom(X)$. The precise statement is below: 

  \begin{theorem}[{\cite[Theorem C]{PSZCAT}}]\label{thmi:rank-one_iff_skewers_iff_QI}
Let $g$ be a semisimple isometry of a CAT(0) space $X$. The following are~\mbox{equivalent}.
\begin{itemize}
\item   $g$ is strongly contracting. 
\item   $g$ acts loxodromically on the curtain model $\X$.
\end{itemize}
\end{theorem}

Recall that a geodesic $\alpha$ in a CAT(0) space is said to be \emph{strongly contracting} if there exists some $D \geq 0$ such that for any ball $B$ disjoint from $\alpha,$ we have $\diam(\pi_\alpha(B)) \leq D.$

\begin{definition} For a CAT(0) space $X$, an element $g \in G$ is said to be \emph{rank-one} on $X$ if it has a strongly contracting axis in $X.$ When $X$ is understood, we omit it and just write rank-one.
\end{definition}

Also, given a CAT(0) cube complex, one can define the \emph{contact graph} introduced by Hagen in \cite{hagen:weak}. Vertices are hyperplanes and you connect two hyperplanes with an edge if they intersect or if their \emph{hyperplane carriers} intersect, where a hyperplane carrier is simply the smallest combinatorially convex set containing the hyperplane. The following theorem is due to Hagen (for instance, see \cite{hagen:weak}, \cite{HagenFacing2022}):

\begin{theorem}
    If $G$ acts geometrically by cubical isometries on a CAT(0) cube complex and $g \in G$ acts loxodromically on the contact graph, then $g$ is rank-one.
\end{theorem}

It's worth noting that the above is not a characterization. For instance, the contact graph for the space $X=[0,1] \times \mathbb{R}$ is coarsely a point, while the group $\mathbb{Z}$, which acts geometrically on $X,$ does contain a rank-one element. Finally, as shown in \cite{hagen:weak}, the contact graph is always a quasi-tree.

\subsection{Special CAT(0) cube complexes} The main lemma we prove in this section (Lemma \ref{lem:special}) is surely known to experts (at least in the context of special compact groups \cite{Haglund2007}), however, surprisingly we were not able to locate it in the literature. Before we state the statement and proof, we recall a few facts.

Let $G$ be any group acting freely on a CAT(0) cube complex $X=X_1 \times X_2.$ Fix a base point $(x_1,x_2) \in X$ and consider $H_1:= \text{Stab}_{G}(X_1 \times \{x_2\})$ and $H_2 =\text{Stab}_{G}(\{x_1\} \times X_2)$. Observe that for each $ g\in H_1 \cap H_2$, we must have $g(x_1,x_2)=(x_1,x_2)$ and since $G$ acts freely on $X,$ we have $g=e$ or $H_1 \cap H_2=\{e\}$. Finally, recall that for any CAT(0) cube complex, a vertex is exactly the intersection of all hyperplane carriers containing it. Now, consider hyperplane carriers in $X$ that are of the form $N(h_i)=X_1 \times \{n(k_i)\}$ for some hyperplane carrier $n(k_i)$ in $X_2$ containing $x_2$. In particular, $X_1 \times \{x_2\}= {\bigcap}_{i=1}^n N(h_i)$ and $\text{Stab}_G(X_1 \times \{x_2\})$ permutes the carriers $\{N(h_i)\}_{i=1}^n.$ Now, none of the above assumes that 
$G$ is a compact special group. The crucial property we need from compact special groups is the following.

\begin{proposition}[{\cite[Proposition 6.7]{Einstein}}]\label{prop:ted} Let $G$ be a group acting freely cocompactly coespecially on a CAT(0) cube complex $X$. There is a finite index subgroup $G' \leq G$ such that for any hyperplane carrier $N(h) \subset X$, we have

\begin{enumerate}
    \item (disjoint translates) $gN(h) \cap  N(h) \neq \emptyset \implies gN(h)=N(h)$, and 
    \item (prohibited inversions) $gN(h)=N(h) \implies gh^+=h^+$ and $gh^-=gh^-$.
\end{enumerate}
    
\end{proposition}

%Hence, a finite index subgroup of $\text{Stab}_G(X_1 \times \{x_2\})$  fixes the set $\{N(h_1),\cdots N(h_n) \}$ pointwise and in particular, the factor action of this subgroup on $X_2$ fixes the combinatorial 1-neighborhood of $x_2.$ In fact, passing to a further finite index subgroup, call it $G_1',$ assures that each of the sets $\{h_1^+, \cdots h_n^+\}$ and $\{h_1^-, \cdots h_n^-\}$ is pointwise fixed. Finally, by the above, $G_1'$ is exactly the intersection ${\bigcap}_{i=1}^n\text{Stab}_G(h_i)'$ where $\text{Stab}_G(h_i)'$ is the index 2 subgroup of $\text{Stab}_G(h_i)'$ fixing $\{h_i^+,h_i^-\}$ pointwise.

\begin{lemma}[Hagen, private communication]\label{lem:special}  Let $G$ be a group virtually acting freely cocompactly coespecially on a CAT(0) cube complex $X=X_1 \times X_2$. There is a finite index subgroup $H\leq G$ satisfying $H=K_1 \times K_2$, where $K_i=\mathrm{Stab}_{H}(X_i)$ and each $K_i$ acts freely cocompactly on $X_i$ for $i \in \{1,2\}.$
    
\end{lemma}

  \begin{proof} First we pass to a finite index subgroup $K \leq G$ that meets conclusions 1 and 2 of Proposition \ref{prop:ted}. We claim that for each $b,b' \in X_2$, we have $\text{Stab}_K(X_1 \times b)=\text{Stab}_K(X_1 \times b')$. To see this, we will show that $\text{Stab}_K(X_1 \times b)=\text{Stab}_K(X_1 \times b')$ for any two adjacent $b,b' \in X_2$ and then the statement follows by induction using the fact that $X_2$ is combinatorially convex. As observed above, for any $b \in X_2$, we have $X_1 \times \{b\}= {\bigcap}_{i=1}^n N(h_i)$ such that $N(h_1)=X_1 \times n(k_i)$ and $n(k_i)$ is a hyperplane carrier containing $b.$ Also, since each $g \in \text{Stab}_K(X_1 \times b)=\text{Stab}_G({\bigcap}_{i=1}^n N(h_i)) $ permutes $N(h_i)$ and since translates of carriers are disjoint (Proposition \ref{prop:ted}), we must have $gN(h_i)=N(h_i)$, in other words, $g(X_1 \times n(k_i))=X_1 \times n(k_i)$ for each $n(k_i) \subset X_2$ containing $b.$ In particular, the factor action of $g$ on $X_2$ fixes the combinatorial 1-neighborhood of $b.$  Therefore, $g \in \text{Stab}_K(X_1 \times b')$ and hence $\text{Stab}_K(X_1 \times b) \subseteq \text{Stab}_K(X_1 \times b')$. Replacing the roles of $b,b'$ shows $\text{Stab}_K(X_1 \times b') \subseteq \text{Stab}_K(X_1 \times b)$ for any two adjacent $b,b' \in X_2$. Thus $\text{Stab}_K(X_1 \times b) = \text{Stab}_K(X_1 \times b')$ for any $b,b' \in X_2,$ and in particular, for any $x_2 \in X_2,$ the factor action of any $K_1=\text{Stab}_K(X_1 \times x_2)$ on $X_2$ is trivial. Similarly, for any $x_1 \in X_1$, the factor action of $K_2=\text{Stab}_K(x_1 \times X_2)$ on $X_1$ is trivial.  Now, fix some $(x_1,x_2) \in X_1 \times X_2$ and observe that the $K$-subgroup $H=\langle K_1, K_2 \rangle $ admits factor action maps $f_i:H \rightarrow \text{Aut}(X_i)$, and by triviality of the action of $K_i$ on $X_{i+1}$ we have $\text{ker}(f_i)=K_{i+1}$, for $i \in\{0,1\}$ mod$(2)$. In particular, the subgroups $K_i$ are normal in $H.$ Since $K_1,K_2$ also have trivial intersection, we have $H=K_1 \times K_2.$ Now, since $H$ acts geometrically on $X$ (as $K_i$ acts geometrically on $K_i)$, the subgroup $H$ must be a finite index subgroup of $G$ concluding the proof.

  \end{proof}

 %Since $\text{Stab}_G(X_1 \times \{x_2\})$ permutes the carriers $\{N(h_i)\}_{i=1}^n$ and translates of each $N(h_i)$ are disjoint, $\text{Stab}_G(X_1 \times \{x_2\})$ must stabilize each carrier. In particular, it fixes the 1-combinatorial neighborhood of $x_2.$ We claim it fixes $x_1 \times X_2 $Observe that since cubical isometries take hyperplanes to hyperplanes and since translates of hyperplanes carriers are disjoint (by assumption), each two distinct translates of $X_1 \times \{x_2\}$ are disjoint and of the form $X_1 \times \{x_2'\}, X_1 \times \{x_2''\}$.

\section{Hyperbolic spaces} In this section, we recall a few well-known statements regarding hyperbolic spaces. Also, we adapt the notion of \emph{curtains}  (introduced in \cite{PSZCAT} for CAT(0) spaces) to the setting of hyperbolic spaces. The proofs of the statements in the following lemma can be found in any standard text on hyperbolic spaces and groups.

\begin{lemma}\label{lem:hyperbolic_facts} Each $E,q > 0$ determine some $E'>E$ such that the following holds true for any $E$-hyperbolic geodesic metric space $X:$
\vspace{2mm}

\begin{enumerate}
    \item (projections) For any geodesic $\alpha,$ there is an assignement $\pi_\alpha:X \rightarrow \alpha$ which assigns to each $x$ the set of its nearest points along $\alpha.$ Further, $\pi_\alpha$ is $(1,E)$-coarsely Lipshitz. Moreover, if $\alpha$ is a $(q,q)$-quasi-geodesic, then $\pi_\alpha$ is an $(E',E')$-coarsely Lipshitz map with $E'=E(E,q).$

    \item (stability/Morse) If $\alpha$ is a $(q,q)$-quasi-geodesic and $\beta$ is a geodesic with the same end points as $\alpha,$ then $\dist_H(\alpha,\beta)<E',$ where $\dist_H$ stands for the hausdorff distance.

    \item (quasi-axes) Each loxodromic isometry $g \in \isom(X)$ admits an $(2E,2E)$-quasi-axis in $X.$

    \item (pre-convexity) For any geodesic $\alpha$, an interval $I \subsetneq \alpha$, and a $(q,q)$-quasi-geodesic $\alpha'$ connecting points $x_1,x_2$ with $\pi_\alpha(x_i) \in I$, we have $\pi_\alpha (\alpha') \subset \cal N_{E'}(I) \cap \alpha.$

    %\item If $\alpha$ is some $(2E,2E)$-quasi geodesic and $\beta:[a_1,a_2]\rightarrow X$ is a geodesic such that each $\beta(t) \in \beta$ satsfies $\beta(t) \notin \cal N_{7E}(\alpha)$, then:
%\begin{itemize}
   % \item (contracting) $\diam(\pi_\alpha(\beta))<5E$, and
    %\item (transversality) There exists $a \in \{a_1,a_2\}$ such that for any $x \in X$, if $\dist (\pi_\beta(x), \beta(a))>2E,$ then $\dist (\pi_\alpha(x), \pi_\alpha(\beta))<2E.$
%\end{itemize}

\end{enumerate}
\end{lemma}

The following definition is a slight modification of Definition A in \cite{PSZCAT}.

\begin{definition}\label{def:hyperbolic_curtains} (curtains, half spaces) Let $X$ be an $E$-hyperbolic geodesic metric space. A \emph{curtain} $h_\alpha$ dual to a geodesic $\alpha:[0,A] \rightarrow X$ is defined to be $\pi_\alpha^{-1}(I)$, where $I=[a,b]$ is an interval of length $6E$ with $I \subset (0,A).$ The \emph{half spaces} corresponding to such a curtain are defined to be $h_\alpha^+=\pi_\alpha^{-1}([0,a))$ and $h_\alpha^{-}:=\pi_\alpha^{-1}((b,A])$. A \emph{curtain} (resp. half space) $h$ is a curtain (resp. half space) dual to some geodesic $\alpha$.
\end{definition}

\begin{remark}\label{remks:few_hyperbolic_remarks} Below, we record a few immediate properties of such curtains.

\begin{enumerate}
    \item Since $\pi_\alpha$ in $(1,E)$-coarsely Lipshitz, the assignment $\pi_\alpha$ is a coarsely well-defined map with diam$(\pi_\alpha(x))<E$.

    \item Observe that for any $x,y\in h^{+}, h^{-}$, we have $\dist(x,y)>\dist(\pi_\alpha(x), \pi_\alpha(y))-E>4E-E=3E$, where the last two inequalities hold as $x,y$ are in $h^+, h^-.$ In particular, the corresponding half spaces $h^+,h^-$ are disjoint (although they might intersect $h$).
    
\item The fact that $\pi_\alpha$ is  $(1,E)$-coarsely Lipshitz assures that any (continuous) path connecting a point $a \in A\subseteq h^+$ to a point $b \in B \subseteq h^-$ must intersect $h.$ For $A,B$ as above, we will say that $h$ \emph{separates} $A,B.$

\item A \emph{chain} of curtains $c$ is defined to be a collection of curtains $c=\{h_i\}$ such that each $h_i \in c$ separates $h_{i-1}$ from $h_{i+1}$ (in particular, every pair of curtains in a chain are disjoint).

\item Finally, every $g \in \isom(X)$ takes curtains to curtains (and half spaces to half spaces) since nearest point projections are $\isom(X)$-equivariant.

\end{enumerate}
\end{remark}

\begin{lemma}\label{lem: relating_distance_to_chains} Let $X$ be a geodesic $E$-hyperbolic space. For any chain $c$ separating two points $x,y \in X$, we have $\dist (x,y)\geq E|c|$. Furthermore, if $c$ is a chain of a maximal length separating $x,y,$ then $E |c| \leq \dist (x,y) \leq 2E|c|$.
\end{lemma}

\begin{proof} The proof is essentially the same as the respective statement in \cite{PSZCAT}, but we provide it for completeness. Let $c=\{h_i\}$ be a chain of curtains separating $x,y$ and $\alpha$ be a geodesic connecting $x,y.$ Since $\{h_i\}$ are pairwise disjoint (as $c$ is a chain), we may pick points $x_i$ on $\alpha$ between $h_{i-1},h_{i}$ for each $i.$ Let $\beta_i$ be the geodesic that $h_i$ is dual to, by item 2 of above remark, we have $\dist(x_i,x_{i+1})>\dist (\pi_{\beta_i}(x_i), \pi_{\beta_i}(x_{i+1}))>E.$ Since $x_i \in \alpha$ and as $\dist(x_i,x_{i+1})>E,$ we have  $E.|c|< \underset {i}{\sum}\dist(x_i,x_{i+1}) \leq \dist(x,y)$. For the ``furthermore" part, given $x,y \in X$, choose a geodesic $\alpha$ connecting $x,y$ and choose intervals $I_i \subset \alpha$ with $\dist (I_i, I_{i+1})=2E$. Taking curtains dual to $\alpha$ at the intervals $I_i$ provides a chain of curtains $c=\{h_i\}$ such that $|c|> \frac{\dist(x,y)}{2E}.$
\end{proof}

The following definition is also inspired by \cite{PSZCAT} (who were inspired by \cite{capracesageev:rank}).

\begin{definition}(skewer, flip) Let $X$ be any metric space and let $A,A^+,A^- \subset X$ be non-empty with $A^+ \cap A^- =\emptyset$ and $X=A \cup A^+ \cup A^-$. For an element $g \in \isom(X),$ we have the following:

\begin{enumerate}
    \item We say that $g$ \emph{skewers} $A$ if $\exists m \in \mathbb{N}$ such that $g^mA^+ \subsetneq A^+$ and $g^mA \cap A=\emptyset$.

\item We say that $g$ \emph{flips} $A^+$ if $gA^{+} \subsetneq A^-$ and $gA \cap A=\emptyset.$

\end{enumerate}

\end{definition}

\begin{lemma} \label{lem:translation__length_via_skewering}Let $X$ be a geodesic $E$-hyperbolic space. An element $w \in \isom(X)$ is loxodromic if and only if it skewers a curtain $h.$ Moreover, if $w^{n}h^+ \subsetneq h^+$ and $w^nh \cap h=\emptyset$, then the stable translation length of $w^n$ is at least $1.$
    
\end{lemma}

\begin{proof} For the backwards direction, suppose that there exists some $n$ such that $w^nh^+ \subsetneq h^+$ and $w^nh \cap h=\emptyset.$ Let $g=w^n$ and notice that as $gh^+ \subsetneq h^+$, we have $g^ih^+ \subsetneq g^{i-1}h^+ \subsetneq \cdots \subsetneq gh^+ \subsetneq h^+$. By the definition of skwering, the above curtains $\{h,gh,\cdots g^ih\}$ form a chain $c$ of cardinality $i+1.$ Hence, by the above lemma, for any $x \in h,$ we have $\dist(x, g^ix) \geq E|c|=E(i+1).$ This proves the backwards direction.  It also proves the ``moreover" part of the statement since $\underset{i \rightarrow \infty}{\lim} \frac{\dist(x,g^ix)}{i} \geq \underset{i \rightarrow \infty}{\lim} \frac{E(i+1)}{i} \geq E>1.$

For the forward direction, suppose that $w$ is loxodromic on $X$. This gives us a $(2E,2E)$-quasi-axis $\alpha$ for $w$. Choose two points $x,y \in \alpha$ along the orbit of $w$ with $\dist(x,y)=30E$ and let $\beta$ be a geodesic segment connecting $x,y$ and suppose that $x$ is closer to $\alpha(0)$ than $y.$ Since $w$ is loxodromic, there exists a power $m$, depending only on $E$ and the translation length of $w,$ such that $\dist(w^m\beta, \beta)>(10)^5E.$ Let $I \subsetneq \beta$ be an interval of length $6E$ with end points $x_1,x_2$ satisfying $\dist(x_1,x)=12E$ and $\dist(x_2,y)=12E$. Let $h$ be the curtain dual to $\beta$ at $I$ and let $h^+$ be the half space given by $\pi_\alpha^{-1}((x_2,y])$, where $(x_2,y]$ is the subgeodesic of $\beta$ connecting $x_2,y.$

\begin{figure}[ht]
   \includegraphics[width=12cm, trim = 2cm 9cm 2cm 3cm]{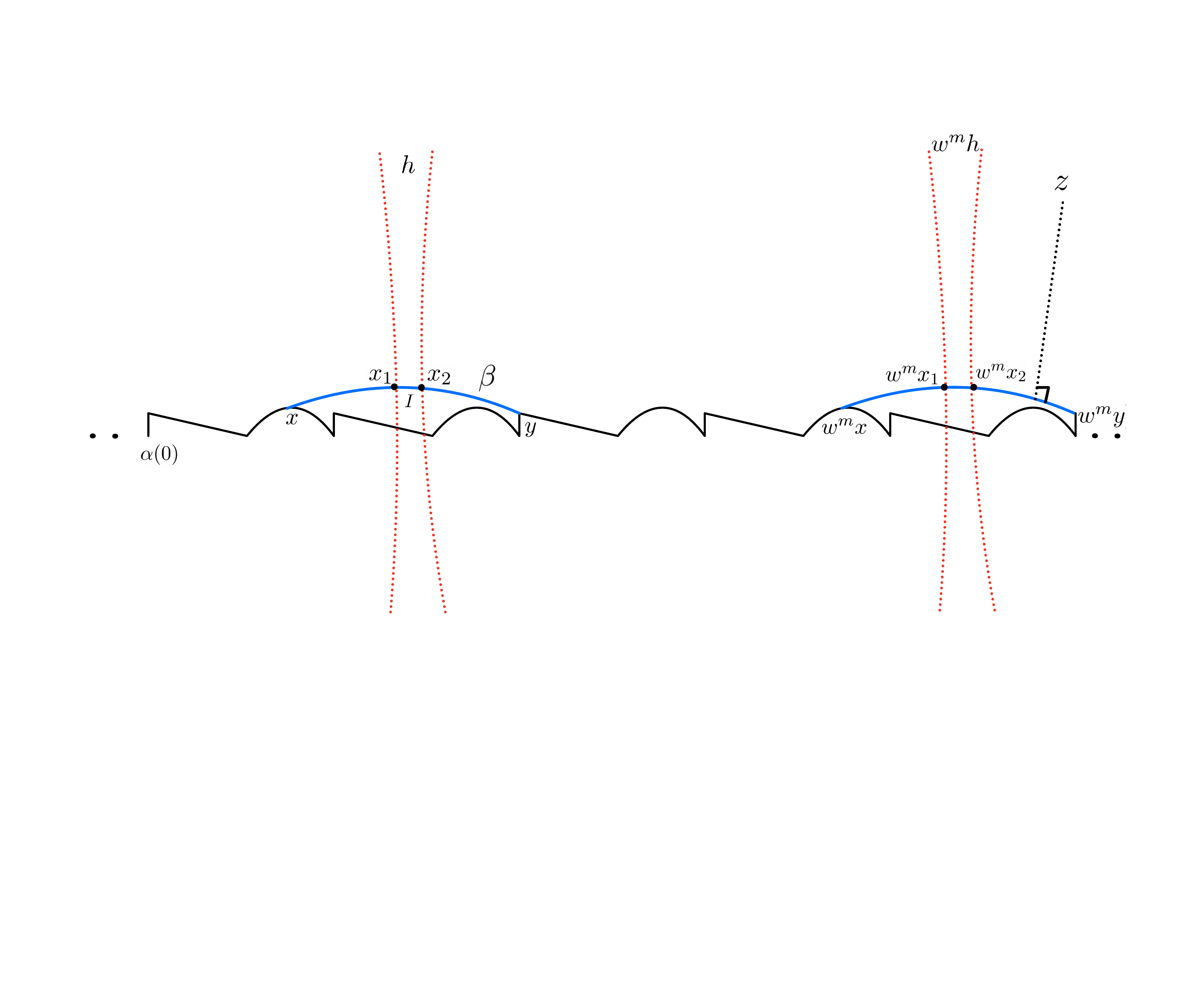}\centering
\caption{Any geodesic connecting $z$ to a point $p \in \beta$ must pass near the projection of $x$ to $w^m\beta$, stay near $\alpha$ until it reaches $p$. In particular, it must pass close to $y$. Hence, the projection of $z$ to $\beta$ is near $y$.} \label{fig:skewering_hyperbolic}
\end{figure}

 We need to show that $w^mh \cap h \neq \emptyset$ and $w^mh^+ \subsetneq h^+.$ To show this, it would suffice to show that $w^m(h \cup h^+)\subsetneq h^+.$ This is immediate by hyperbolicity, namely, if $z \in w^m(h \cup h^+),$ then its projection to $w^m\beta$ lies between the points $w^mx_1, w^my$ along the geodesic $w^m\beta$, see Figure \ref{fig:skewering_hyperbolic}. Using hyperbolicitiy of $X$, any geodesic $\gamma$ connecting $z$ to a point in $\beta$ must pass $2E$-close to $y$, in particular, the projection of $x$ to $\beta$ must be within $2E$ of $y$ which shows that $x \in h^+$ concluding the proof.

%We leave the forward direction as an exercise for the reader (but the reader should also be aware that due to lack of properness, there need not exist a geodesic line that is close to $\langle g \rangle . x_0 $ for any $x_0$. However, the statement can be proven by considering geodesics in $\hull (\langle g \rangle . x_0)=\cup [g^ix_0,g^jx_0]$ where $[x,y]$ is the collection of all geodesics connecting $x,y$ and the union is taken over all $i,j \in \mathbb{N}$).
\end{proof}

\begin{remark}\label{rmk: skewering_integer}
    Observe that by the proof of Lemma \ref{lem:translation__length_via_skewering} above, when $w$ is loxodromic on an $E$-hyperbolic space, the integer $m$ required to skewer depends only on $E$ and $\tau_X(w)$.
\end{remark}

The following observation is essentially due to Caprace and Sageev \cite{capracesageev:rank}, here, we merely remark that their observation works in any geodesic hyperbolic space with curtains replacing hyperplanes.

\begin{corollary}\label{cor: flip_then_skewer}
    Let $h$ be a curtain in a geodesic hyperbolic space $X$ and let  $g_1,g_2 \in \isom(X)$. If $g_1$ flips $h^+$ and $g_2$ flips $h^-$ then $g_2g_1 (h \cup h^+) \subsetneq h^+$. In particular, $g_2g_1$ skewers $h$ and $\tau_X(g_2g_1) \geq 1.$ 
\end{corollary}

\begin{proof}
    The proof of the claim that $g_2g_1 (h \cup h^+) \subset h^+$ proof is exactly the same as the proof of the double-skewering statement in the introuduction section of \cite{capracesageev:rank}, hence, we refer the reader to that paper. The conclusion that $\tau_X(g_2g_1) \geq 1$ follows from the previous Lemma \ref{lem:translation__length_via_skewering}.
\end{proof}

\section{Hierarchically hyperbolic spaces}\label{sec:HHS_prelim}

We conclude the preliminaries with a brief discussion of \emph{hierarchically hyperbolic spaces} (HHSes). In light of the complexity of the definition of an HHS, we will only highlight the features that we require for our analysis.  See \cite{SistoWhatIs} for an excellent overview of the theory. Roughly, an HHS is a pair $(X, \mathfrak S)$ where $X$ is a geodesic metric space and $\mathfrak S$ is set indexing a family of uniformly hyperbolic spaces $\calC U$ for each $U \in \mathfrak S$ such that the following hold:

\begin{enumerate}

\item For each \emph{domain} $W \in \mathfrak S$ there exist a $\lambda$-hyperbolic space $\calC W$ and a coarsely surjective $(\lambda,\lambda)$-coarsely Lipschitz map $\pi_W: X \rightarrow \calC W,$ with $\lambda$ independent of $W$.

    \item The set $\mathfrak S$ satisfies the following:
    \begin{itemize}
        \item It is equipped with a partial order called \emph{nesting}, denoted by $\nest$. If $\mathfrak S \neq \emptyset,$ the set $\mathfrak S$ contains a unique $\nest$-maximal element $S.$
        
        \item It has a symmetric and  anti-reflexive relation called \emph{orthogonality}, denoted by $\perp$. Furthermore, domains have \emph{orthogonal containers}: if $U \in \mathfrak S$ and there is a domain orthogonal to $U$ then there is a domain $W$ such that $V \nest W$ whenever $V \perp U.$
        
        \item For any distinct non-orthogonal $ U, V \in \mathfrak S$, if neither is nested into the other, then we say $U,V$ are \emph{transverse} and write $U  \pitchfork V.$ 
            \item There exists a map $\theta:[0,\infty) \rightarrow [0,\infty)$ such that for any $D \geq 0,$ if $\dist_X(x,y) \geq \theta(D),$ then $\dist(\pi_U(x),\pi_U(y))>D$ for some $U \in \mathfrak S.$ We will refer to this property as the \emph{uniqueness} property. \label{def:uniqueness}
    
        \end{itemize}

        \item There exists an integer $N$ called the \emph{complexity} of $X$ such that whenever $U_1,U_2,\cdots U_n$ is a collection of pairwise non-transverse domains, then $n \leq N.$
        
        \item If $U  \sqsubset V$ or $U  \pitchfork V,$ then there exists a set $\rho^U_V$ of diameter at most $\lambda$ in $\calC V.$ Further:
        
        \begin{itemize} \item If $U \pitchfork V$ and $x \in X$ with $\dist_U(\pi_U(x), \rho^V_U)> \lambda,$ then $\dist_V(\pi_V(x), \rho^U_V)< \lambda.$  

        \item If $U \sqsubseteq V,$ then $\dist_W(\rho^U_W, \rho^V_W)< \lambda$ whenever $V \sqsubsetneq W$ or $V \pitchfork W$ and $W \notperp U$.
        \end{itemize}

        \item If $U  \sqsubset V$, there is a map $\rho^V_U:\calC V \rightarrow \calC U$ which satisfies the following: For $x \in X$ with $\dist_{V}(\pi_V(x),\rho^U_V)>\lambda,$ we have $\dist_U(\pi_U(x),\rho^V_U(\pi_V(x))< \lambda.$

\end{enumerate}

\begin{lemma}(Bounded geodesic image)\label{lem:BGIA}  If $U  \sqsubset V$ and $\gamma \in \calC V$ is a geodesic with $\dist_{V}(\rho^U_V,\gamma)>\lambda$, then diam$(\rho^V_U(\gamma))\leq \lambda.$ Furthermore, for any $x,y \in X$ and any geodesic $\gamma$ connecting $\pi_V(x),\pi_V(y)$, if $\dist_{V}(\rho^U_V,\gamma)>\lambda$, then $\dist_U(\pi_U(x), \pi_U(y)) \leq \lambda.$
\end{lemma}

Another standard fact about HHses is the existence of $\emph{hiierachy paths}$. See \cite{HHS2}.

\begin{lemma}
    Each pair of points $x,y$ in an $HHS$ are connected by some $(\lambda, \lambda)$-quasi-geodesic $\alpha_{x,y}$ such that for each $U \in \mathfrak S,$ we have $\pi_U(\alpha_{x,y})$ is an unparametrised $(\lambda, \lambda)$-quasi-geodesic.
\end{lemma}

For two points $x,y \in X,$ it is standard to use $\dist_U(x,y)$ to denote $\dist_{\calC U}(\pi_U(x), \pi_U(y)),$ and similarly for subsets of $X.$ For a subset $A \subseteq X,$ we will also use $\diam_U(A)$ to denote the diameter of the set $\pi_U(A) \subset \calC U$. Further, for a constant $D \geq 0$, the set $\text{Rel}_D(x,y)$ denotes the collection of all domains $U$ with $\dist_U(x,y) \geq D.$

\begin{notation}  We will use the following notations:

\begin{enumerate}
    \item For domains $U_1,U_2 \in \mathfrak S$, we denote the $\rho$-map (if it exists) from $U_1$ to $U_2$ by $\rho^1_2$.
\item For two points $x,y$ along a geodesic $\beta$, we denote the subgeodesic of $\beta$ connecting $x,y$ by $[x,y]_\beta.$

\end{enumerate}

\end{notation}

Now, we define \emph{hierachichally hyperbolic groups}. We are following the definition given in \cite{PetytSpriano20} as it appears to be the most compact, but the notion was originally introduced in \cite{HHS1} and \cite{HHS2}. 

\begin{definition}(HHG) Let $G$ be a finitely generated group whose Cayley graph $X$ admits an HHS structure $(G, \mathfrak S).$ The group $G$ is said to be an HHG if:

\begin{itemize}
    \item $G$ acts cofinitely on $\mathfrak S$ and the action preserves the three relations. For each $g \in G$ and each $U \in \mathfrak S$, there is an isometry $g:\calC U \rightarrow \calC gU$ and these isometries satisfy $g.h=gh.$
\item There is an equivariance of the form $g \pi_U(x)=\pi_{gU}(gx)$ and $g \rho^V_U=\rho^{gV}_{gU}$ for all $g,x \in G$ and all $U,V \in \mathfrak S$ where $\rho^{U}_V$ is defined.

\end{itemize}
    Finally, the group $\Aut(\mathfrak S)$ consists of all elements $g$ which satisfy items 1 and 2 of the above definition with the exception that the action on $\mathfrak S$ need not be co-finite.
\end{definition}

\begin{remark}(bounded domains)\label{rmk:bounded_domains}
   As remarked after Definition 2.3 in \cite{PetytSpriano20}, if $G$ is an HHG, theb -- using cofiniteness of the action of $G$ on $\mathfrak S$ -- there exists a constant $K=K(G)$ such that if $\pi_U(G)$ is bounded, then the diameter of $\pi_U(G)$ is at most $K$.
\end{remark}

The following is Theorem 5.1 in \cite{PetytSpriano20}.
\begin{theorem}[{\cite[Theorem 5.1]{PetytSpriano20}} ]\label{thm:eyries} Let $(G, \mathfrak S)$ be an HHG and let $H \leq G$ be any subgroup. There is a finite $H$-invariant set $\cal E(H)=\{S_1,\cdots S_n\}$ of pairwise orthogonal domains with unbounded $H$-projection such that every $U$ with $\pi_U(H)$ unbounded is nested in some $S_i.$ Moreover, if $H$ is finitely generated and infinite, then $\cal E$ is non-empty.
\end{theorem}

The set of domains $\{S_1, \cdots S_n\}$ in Theorem \ref{thm:eyries} above are called the \emph{eyries} of $H$. The first part of the following Corollary \ref{cor:singlton} follows from the statement and proof of Corollary 4.7 in \cite{PetytSpriano20}. The second part is Remark 4.9 in \cite{PetytSpriano20} or Theorem 14.3 in \cite{HHS1} (see also Theorem 1.3 in \cite{contractingrandom}  and Theorem 6.3 in \cite{RST18}).

\begin{corollary}[{\cite{PetytSpriano20}}]\label{cor:singlton} Let $G$ be an $HHG.$ The eyries of $G$ is a singlton $\{S\}$ if and only if $G$ contains a Morse element. In this case, the action of $G$ on $\calC S$ is acylindrical and every loxodromic on $\calC S$ is Morse in $G.$
\end{corollary}

In the context of CAT(0) spaces, an element is Morse if and only if it's rank-one.

\begin{theorem} [{\cite{charneysultan:contracting}}] Let $g$ be a simesimple isometry of a proper CAT(0) space, then $g$ is Morse if and only if $g$ is rank-one.
    
\end{theorem}

One of the main powers of the existence of such eyries is that they are intimately connected to the algebraic structure of the group and can be used to detect whether $G$ is virtually Abelian (see Corollaries 4.2 and 4.3 in \cite{PetytSpriano20}). In particular, we have the following.

\begin{theorem}[{\cite[Corollary 4.3]{PetytSpriano20}}]\label{thm:Characterization_Abelian} Let $G$ be an HHG with eyries $\{S_1,\cdots S_n\}$. If $G$ is not virtually Abelian, then there exists an eyry $S_i$ such that $\calC S_i$ is not quasi-isometric to a line.
    
\end{theorem}

\begin{remark}\label{rmk:constants} Below, we set up some conventions regarding the constants that we will use throughout the paper.

\begin{enumerate}

\item  Let $G$ be an HHG, for the rest of the paper, we fix a single Cayley graph $X$ for $G$ and we think of $X$ as an HHS on which $G$ acts properly cocompactly. In particular, this $X$ comes with a constant $\lambda$ and $G$ comes with a constant $K$ as in Remark \ref{rmk:bounded_domains}.

%\item Let $h$ be a curtain in $\calC U$ dual to a geodesic $\alpha$ at interval $I$ and let $H^+, H^-$ its $U$-half spaces in $X.$ The set $H^+$ need not be median convex in $X$, however, we can consider the median hull of $H^+$ which is formed by iteratively joining end points of $H^+$ by geodesics (as in Definition \ref{def:median_hull}) forming $J^1(H^+), J^2(H^+),\cdots J^N(H^+)=J^{N+1}(H^+)=\hull(H^+).$ Since median paths of $X$ define uniform quality re-paramterized quasi-geodesics in $\calC U$, item 4 of Lemma \ref{lem:hyperbolic_facts} implies that $\pi_\alpha(\pi_U(J^i(H^+)))$ is within $i.\lambda'$ of an end point of $I;$ where $\lambda'$ is a constant that depends only on $\lambda$ ( where $\lambda$ is the constant given in the definition of the HHS). In words, the nearest point projection of each point of $\pi_U(\hull(H^+))$ to $\alpha$ is uniformly close to an end point of $I.$ This remark will be crucial for Lemma \ref{lem:hyperplane_in_curtain}.

\item Since each $\calC U$ is $\lambda$-hyperbolic, the constant $\lambda$ determines a constant $\lambda' > \lambda$ as in item 4 of Lemma \ref{lem:hyperbolic_facts}. For an HHS $(X,\mathfrak S)$ given with constant $\lambda$ as above, we take a constant $E=\text{max}\{K,N.\lambda'\}>N\lambda$ and we will call this \emph{the HHS constant.} It is clear that all the axioms listed are still true with $E$ replacing $\lambda$ since $E> \lambda'>\lambda$. In particular, all the hyperbolic spaces are $E$-hyperbolic and the curtains taken in Definition \ref{def:hyperbolic_curtains} are defined via intervals $I$ of length $6E$ for this larger $E.$ For the rest of the paper, we will work with the constant $E$ while keeping in mind that the axioms hold true for the smaller constant $\lambda$ as this will come in handy sometimes. 

\item When $X$ is a finite dimensional CAT(0) cube complex with a factor system, then it's an HHS and the complexity constant of the HHS $N$ satisfies $N \geq v$ where $v$ is the dimension of $X.$

\item Let $h$ be a curtain in $\calC U$ dual to a geodesic $\alpha$ at interval $I$ and let $H^+, H^-$ its $U$-half spaces in $X.$ Suppose further that $X$ is a CAT(0) cube complexes; the set $H^+$ need not be median convex in $X$, however, we can consider the median hull of $H^+$ which is formed by iteratively joining end points of $H^+$ by geodesics (as in Definition \ref{def:median_hull}) forming $J^1(H^+), J^2(H^+),\cdots J^N(H^+)=J^{N+1}(H^+)=\hull(H^+).$ Since geodesics of $X$ define uniform quality re-paramterized quasi-geodesics in the $\lambda$-hyperbolic space $\calC U$, for any $x \in \hull(H^+)-H^+$, item 4 of Lemma \ref{lem:hyperbolic_facts} implies that $\pi_\alpha(\pi_U(x)))$ is within $i.\lambda'$ of an end point of $I;$ where $\lambda'$ is a constant that depends only on $\lambda$, and $\lambda$ is the constant given in the definition of the HHS. In words, the nearest point projection of each ``new point" in $\hull(H^+)$ to $\alpha$ is uniformly close to an end point of $I.$ In light of the previous two items and our choice of $E \geq N \lambda',$ we get that for any $x \in \hull(H^+)-H^+,$ the nearest point projection of $\pi_U(x)$ to $\alpha$ is within $E$ of an end point of $I.$ The remark will be important to the proof of Lemma \ref{lem:hyperplane_in_curtain}.

\end{enumerate}

\end{remark}

\begin{lemma}[{\cite[Lemma 6.3 and 6.7]{Durham2017-ce}}]\label{lem:Big_Sets} Let $(X, \mathfrak S)$ be an HHS and let $x \in X.$ Every non-torsion element $g \in \Aut(\mathfrak S)$ determines a (finite) collection of pairwise orthogonal domains $\vbig(g)=\{U_i\}$ such that $\pi_{U_i}(\langle g \rangle. x)$ has infinite diameter and has finite diameter for every other domain. Further, there exists a constant $M=M(\mathfrak S)$ such that $g^MU_i=U_i$ for each $i$.
\end{lemma}

In order to simplify certain statements, we introduce the following definition.

\begin{definition}(active) Continuing with the notation from the above lemma, we will say that $g$ is \emph{active} over $U$ if $U \in \vbig(g)$ and $gU=U.$

\end{definition}

 The following lemma is an immediate consequence of the Behrstock inequality (item 4 in the definition of an HHS Section \ref{sec:HHS_prelim}).

\begin{lemma} Let $U_i$ be a collection of domains for $i \in \{1,2,3\}$ such that for $i \in \{1,2\}$, the domains $U_i,U_{i+1}$ are transverse and $\dist_{U_2}(\rho^1_2, \rho^3_2)>4E$, then $\{U_i\}_{i=1}^3$ is a collection of distinct pairwise transverse domains.
\end{lemma}

\begin{proof} Pick points $x_i \in X$ such that such that $\dist_{U_i}(x_i, \rho^2_i)>3E$, and $\dist_{U_2}(x_2,\rho^i_2) \asymp 2E$ for $i\neq 2.$ By the Behrstock inequality (item 4 of the HHS definition Section \ref{sec:HHS_prelim}), we have $\dist_{U_2}(\rho^1_2,x_i)>E$ for $i \neq 1$. Thus, we have $\dist_{U_1}(\rho^2_1,x_i)<E$ and by the triangle inequality we get $\dist _{U_1}(x_2,x_3)<2E.$ Since $\dist_{U_3}(x_2,x_3)>3E$, we conclude that $U_1 \neq U_3.$ As $\dist_{U_2}(\rho^1_2, \rho^3_2)>3E$, the domains $U_1,U_3$ can't have a nesting relation, and by partial realization, $U_1,U_3$ can't be orthogonal. Hence, $\{U_i\}$ is a collection of distinct pairwise transverse domains.
\end{proof}
Applying the above lemma inductively gives the following consequence.

\begin{corollary} (local-to-global for transverse domains)\label{cor:local-to-global} If $\calU=\{U_i\}_{i=1}^k$ satisfies that $U_i,U_{i+1}$  are transverse and $\dist_{U_i}(\rho^{i-1}_i, \rho^{i+1}_i)>4E$ for all $1 \leq i \leq k,$ then $\calU$ consists of distinct pairwise transverse domains.
\end{corollary}

%\begin{lemma}\label{lem:short_translation} If $G$ is an HHG, then there exists some $\epsilon=\epsilon(G)>0$ such that for each non-torsion element $a \in G,$ there exists $U \in Big(a)$ such that $\tau_U(g)>\epsilon.$
    
%\end{lemma}

The following is a (strong) variant of the standard passing-up argument proven recently by Durham (compare with Lemma 2.5 \cite{HHS2} and Lemma 5.4 in \cite{PetytSpriano20}). It is worth noting that Durhams actual statement is significantly stronger than Lemma \ref{lem:strong_passing} below, however, the following version is all we need for this article.

\begin{lemma}[{\cite[Proposition 3.5]{Durham2023}}]\label{lem:strong_passing} For any $K_1>K_2>50E,$ there exists an integer $P=P(K_1)$ so that for any $x,y \in X$ if $\calV \subset \mathrm{Rel}_{K_1}(x,y)$ with $|\calV|>P$, then there exist $W \in \mathrm{Rel}_{K_2}(x,y)$ and $U_1,U_2,U_3 \in \calV$ with $U_i \nest W$ and $\dist_W(\rho^{U_i}_W, \rho^{U_j}_W)>K_2,$ for each $ i \neq j$, where $i,j \in \{1,2,3\}.$
\end{lemma}

The above lemma requires the existence of two points $x,y$ and domains $U_i \in \mathrm{Rel}_{K_1}(x,y)$ in order to produce a domain $W$ higher in the nesting where $x,y$ are still far. However, the existence of such points $x,y$ is automatic if the domains $\{U_i\}$ are pairwise transverse:

\begin{lemma}\label{lem:strong_passing_up_consequence}
    For each $K_1>K_2>50 E$, there exists some integer $P=P(K_1)$ such that if $U_1,U_2, \cdots U_P$ are domains where $U_i,U_{i+1}$ are transverse for each $1 \leq i\leq P-1$ and $\dist_{i}(\rho^{i-1}_i, \rho^{i+1}_{i})>K_1$, then there exist domains $\{V_k\}_{k=1}^3 \subseteq \{U_1, \cdots U_P\}$ and a domain $W$ with $V_k \nest W$ and $\dist_W(\rho^{V_r}_W,\rho^{V_s}_W)>K_2$ for all $r \neq s$ with $r,s \in \{1,2,3\}.$
\end{lemma}

 \begin{proof} This is an immediate application of Lemma \ref{lem:strong_passing}. Namely, choose points $x,y \in X$ such that the projections $\pi_{U_1}(x), \pi_{U_P}(y)$ are more than $E$ away from $\rho^2_1$ and $\rho^{P-1}_P$ respectively (this is possible since for each $Z \in \mathfrak S,$ the map $\pi_Z$ is coarsely surjective).  Item 4 of the definition of an HHS assures that $U_1, \cdots U_P \in \text{Rel}_{K_1}(x,y)$. The passing-up Lemma \ref{lem:strong_passing} provides domains $\{V_k\}_{k=1}^3$ and a domain $W$ with $V_k \nest W$ and $\dist_W(\rho^{V_s}_W, \rho^{V_t}_W)>K_2$ for all $s\neq t$ with $1 \leq s,t \leq 3.$

\end{proof}

The following lemma states that it is generally easy to produce a pair of transverse domains.

\begin{lemma}\label{lem:Tranverse_in_full_subgroup}
    Let $H$ be a subgroup of an HHG with eyries given by $\mathcal E(H)=\{S_1,S_2,\cdots S_n\}$ such that $H.S_j=S_j$ and $U \nest S_j$ for some $j$. There exists a pair of transverse domains in $HU=\{hU|h \in H\}$.
\end{lemma}

\begin{proof}
Since the $H$-orbit in each $S_j$ is infinite, if $U \nest S_j$ for some $j,$ then there exists $h \in H$ with $\dist_{S_i}(\rho^U_{S_i}, h\rho^U_{S_i})=\dist_{S_i}(\rho^U_{S_i}, \rho^{hU}_{hS_i})=\dist_{S_i}(\rho^U_{S_i},\rho^{hU}_{S_i} )>31E$ and hence the domains $U,hU$ are transverse (for instance, by Lemma 2.17 in \cite{Abbot-Ng-Spriano}).
\end{proof}

\begin{lemma}\label{lem:producing_first_transverse_pair} Let $H<G$ be a subgroup with eyries given by $\mathcal E(H)=\{S_1,S_2,\cdots S_m\}$ such that $HS_i=S_i$ for all $i$. If there exists $a \in H$ that is active over $U \nest S_i$ for some $i$, then the collection $\underset{i=0}{\cup}^{N+1}\{T^iU\}$ contains a pair of transverse domains, where $N$ is the complexity of the HHS.

\end{lemma}

\begin{proof} 

Let $S_i$, $U$ and $a$ be as in the statement. Consider the collection of domains $\cal U=\{U\} \cup \{TU\} \cdots \cup \{T^{N+1}U\}$. We claim that $\calU$ contains a pair of transverse domains. If $\calU$ is $H$-invariant, then $\calU=HU$ and hence by Lemma \ref{lem:Tranverse_in_full_subgroup}, $\calU$ contains a pair of transverse domains. Suppose $\cal U$ is not $H$-invariant, this gives $T\calU \nsubseteq \calU$ and therefore, $\{T^kU\} \nsubseteq \cup_{i=0}^k\{T^iU\}$ for each $k \in \{1, \cdots N+1\}.$ In particular, we have $| \{U\} \cup \{TU\} \cup \cdots \{T^{N+1}U\}|>N+1$ which implies that $\{U\} \cup \{TU\} \cdots \{T^{N+1}U\}$ contains a pair of transverse domains.

\end{proof}

\section{Flipping and skewering half spaces in HHSes}

 The following definition is inspired by \cite{PSZCAT} (see also \cite{Zalloum23_injectivity}).

\begin{definition}
    Let $X$ be an HHS, $U \in \mathfrak S$ and let $\alpha$ a geodesic in $\calC U$. A \emph{$(U,\alpha)$-curtain}, denoted $H_{U,\alpha},$ is defined to be $\pi_U^{-1}(h_\alpha)$ for a curtain $h_\alpha$ dual to $\alpha$ in $\calC U.$ Similarly, the sets $H_{U,\alpha}^+:=\pi_U^{-1}(h_\alpha^+)$, $H_{U,\alpha}^-:=\pi_U^{-1}(h_\alpha^-)$ are said to be $(U,\alpha)$-\emph{half spaces}. A \emph{$U$-curtain}, denoted $H_U$, is a $(U,\alpha)$-curtain for some geodesic $\alpha.$ We define \emph{$U$-halfspaces} $H_U^+, H_U^-$ similarly.
    \end{definition}

It's immediate from the definition that for each $U \in \mathfrak S$, we have $X=H_U^+ \cup H_U \cup H_U^-.$

\begin{remark}\label{rmks: HHS_curtains} We record a few essential remarks regarding $U$-curtains and half spaces.

\begin{enumerate}
    \item The sets $H_U^+, H_U^-$ are disjoint and separate: Observe that as $\pi_U$ is $(E,E)$-coarsely Lipshitz, for each $U \in \mathfrak S$ and each $x \in H_U^+= \pi_U^{-1}(h_\alpha^+)$, $y \in H_U^-=\pi_U^{-1}(h_\alpha^-)$ we have $3E<\dist_U(x,y) \leq E\dist(x,y)+E$ where the first inequality holds by Remark \ref{remks:few_hyperbolic_remarks}. Hence $\dist(x,y) \geq 2$ for all such $x,y \in X$ that are in opposite $U$-half spaces, and in particular, any (continuous path) connecting $x,y$ must intersect $H.$

    %\item Curtains and half spaces are path connected while $U$-curtains and $U$-half spaces are only coarsely path connected (but none of the aforementioned sets is median convex, in general). To see that curtains are path-connected for instance, let $h$ be a curtain dual to some $\alpha$ at $I$ in some $\calC U$ and $x_1,x_2 \in h,$ then each of the sets $\pi_\alpha(x_i)$ has a point $y_i \in I$. If $\beta_i$ is a geodesic connecting $x_i,y_i,$ then each point in $\beta_i$ lives in $h$ (as for each point $z$ in $\beta_i$, the point $y_i \in \pi_\alpha(z)$). Hence, the path $\beta_1 \cup [y_1,y_2]_\alpha \cup \beta_2$ is a path connecting $x,y$ contained in $h.$ A similar argument shows that $h^+, h^-$ are path connected. Since $X$ is path-connected and the maps $\pi_U$ are all coarsely surjective, we get that each $H_U,H_U^+$ and $H_U^-$ are coarsely path-connected.

    \item  Finally, since $g:\calC U \rightarrow g \calC U$ is an isometry and $g\pi_U(x)=\pi_{gU}(gx),$ we have $g.H_{U}=H_{gU}$ (where if $\alpha \subset \calC U$ is the geodesic defining $H_U,$  then $g \alpha \subset gU$ is the geodesic defining $H_{gU}$).

\end{enumerate}

\end{remark}

%\begin{remark} The definition of $U$-curtains is inspired by PSZ23 where the authors develop a similar notion of curtains that hold in the broader class of \emph{locally quasi-cubical spaces}. On that note, and as we do in that article, one can define curtains in an $E$-hyperbolic space in a slightly different way assuring that the resulting curtains and half spaces are $E$-quasi-convex by simply considering the convex hull of $\pi_\alpha^{-1}(I)$ with $I \subsetneq \alpha.$. Similarly, $U$-curtains and half spaces can also be assured to be $E$-median-convex (equivalently, hierarchichally quasi-convex), where $E$ is a constant depending only on the HHS. While doing so does result in $U$-curtains having much stronger (and cubical-like) properties, we do not do that in the current article as it's not needed.
%\end{remark}

\begin{lemma}\label{lem:hyperplane_in_curtain}
    Let $X$ be a finite dimensional CAT(0) cube complex given with an HHS structure $(X,\mathfrak S)$ and let $U \in \mathfrak S.$ Every $U$-curtain $H_U$ contains a cubical hyperplane $h.$ Furthermore, $H^+ \subsetneq h^+$, $H^- \subsetneq h^-$, $h^+ \subseteq H \cup H^+$ and $h^-\subseteq H \cup H^-$
\end{lemma}

\begin{proof} Let $v$ be the dimension of the CAT(0) cube complex. Recall that the complexity of the HHS $N$ satisfies $N \geq v$ and the HHS constant $E=N\lambda'>N\lambda$ (Remark \ref{rmk:constants}). Let $\alpha$ be a geodesic in $\calC U$ with $H_U=H_{U,\alpha}$ and let $I$ the interval defining $H_{U,\alpha}.$ For simplicity, we will denote $H=H_U=H_{U,\alpha}$ and similarly for the respective half spaces. First, we show that $\dist_X(\hull(H^+), \hull(H^-))>1$ from which the statement will easily follow. Recall from Definition \ref{def:median_hull} that $\hull(H^+)$ is obtained by iteratively adjoining median intervals (which are simply combinatorial geodesics here) starting and ending on $H^+, J^1(H^+), \cdots J^{v}(H^+)=J^{v+1}(H^+)=\hull(H^+)$. Hence, for any $x \in \hull(H^+)-H^+$, each point of the set $\pi_\alpha(\pi_U(x))$ is within $N.\lambda'<E$ of an endpoint of $I$ (see item 4 of Remark \ref{rmk:constants}). An analogous statement holds for $H^-$ with respect to the other end point of $I$. Since $|I|=6E,$ and $\pi_\alpha$ is $(1,E)$-coarsely Lipshitz, we have $\dist_U(\pi_U(x), \pi_U(y)) \geq 6E-2E-E=3E$ for any $x \in \hull(H^+), y \in \hull(H^-).$ Also, as $\pi_U$ is $(E,E)-$coarsely Lipshitz, we have $\dist_X(x,y)>\frac {3E-E}{E}=2.$ This shows that $\dist_X(x,y)>2$ for any $x \in \hull(H^+)$ and $y \in \hull(H^-).$ From here, one can argue in a few different ways, for instance, pick a point $a \in \hull(H^+)$ and consider its image under the combinatorial nearest point projection $P_{\hull(H^-)}:X\rightarrow \hull(H^-)$, call such a projection $a'.$ Now, let $a''$ be the image of $a'$ under the combinatorial nearest point projection to $\hull(H^+).$ By the above, we know that $\dist(a',a'')\geq 2,$ and hence at least two cubical hyperplanes separate $a',a''$, call them $h_1,h_2.$ Since each $h_i$ separate $a'$ from its projection $a'',$ each $h_i$ must separate $a'$ from $\hull(H^+)$ (see Definition \ref{def:combinatorial_projections}) and in particular, each $h_i$ separates $a'$ from the set $\{a, a''\}$. Therefore, each $h_i$ separates the points $a$, $a'$ and hence separates $a$ from $\hull(h^-).$ This shows that $h_1,h_2$ are strictly between $\hull(H^+), \hull(H^-)$ concluding the proof. The ``furthermore" part is immediate.
\end{proof}

The following lemma shows how transversaility combined with the bounded geodesic image theorem can be used to flip $U$-half spaces in $X$.

\begin{lemma} [an HHS flipping lemma]\label{lem:HHS_flip}  Let $\alpha$ be a geodesic in $\calC W$ connecting $x_1 \in \rho^U_W$ to $x_2 \in \rho ^V_W$ for $U,V \nest W$ with $\dist_W(x_1, x_2)>20E$ and let $b \in \vbig(V)$. If $h$ is a curtain dual to $\alpha$ at interval $I\subset \alpha$ with $\rho^U_W \in h^+$ and $\dist_W(x_i, I) \geq 6E$ for each $i \in \{1,2\}$, then $\exists m=m(X, \tau_V(b)) \in \mathbb{N}$ such that for any $n \geq m$, the element $b^n$ flips $H_{W,\alpha}^+$ and $\pi_W(b^nH_{W,\alpha}^+) \subsetneq h_\alpha^-$.

\end{lemma}

\begin{proof} For simplicity, we will denote $h_{\alpha}$ simply by $h$. Similarly, we denote $H_{W, \alpha}$, $H_{W, \alpha}^+$, $H_{W, \alpha}^-$ by $H,H^+$, and $H^-.$ First, observe that since $\dist_W(x_1,x_2 )>20E$ and $x_1 \in \rho^U_W, x_2 \in \rho^V_W$, the domains $U,V$ must be transverse (for instance, by Lemma 2.17 in \cite{Abbot-Ng-Spriano}). 
\begin{claim*} $\dist_V(x, \rho^U_V)<E$ for any $x \in \pi_V(H^+ \cup H)$. In particular, $\diam_V(\pi_V(H^+ \cup H))<2E$.
\end{claim*}

\begin{claimproof} First, by definition of $H^+ \cup H$, each point $x \in H^+ \cup H$ projects far from $\rho^V_W$ in $\calC W$, hence, by the bounded geodesic image theorem gives us that $\dist_V(x,y)<E$ for any $x,y \in H^+ \cup H$. Thus, it suffices to find one point $z \in H^{+} \cup H$ with $\dist_V(z,\rho^U_V)<E.$ If we pick points $z,z' \in X$ such that $\pi_W(z')$ coarsely lives in $I \subset \alpha$ and $\dist_U (z, \rho^V_U)>E$, then we have $\dist_V(\pi_V(z), \rho^U_V)<E,$ and $\dist_W (z, \rho^U_W)<E$. In particular, $\pi_W(z) \in h^+$, and $\dist_V(\pi_V(z), \rho^U_V)<E$ which proves the desired claim.    
\end{claimproof}

 Since $b \in \vbig(V)$, using Lemma \ref{lem:Big_Sets}, there exists some $M=M(\mathfrak S)$ such that $b^MU=U$. Since $\diam_V(\pi_V(H^+ \cup H))<2E$, there exists some $m=m(\tau_V(b), \mathfrak S)$, such that $$\dist_V(b^{mM}\pi_V(H^+ \cup H), \pi_V(H^+\cup H))>2E.$$ Since $\dist_V(\pi_V(H^+ \cup H), \rho^U_V)<E$, we have $\dist_V(\pi_V(b^{mM}(H^+ \cup H), \rho^U_V)=\dist_V(b^{mM}\pi_V(H^+ \cup H), \rho^U_V)>2E-E=E$ where the first equality holds because $b^MV=V$ and the fact that $g\pi_Z(z)=\pi_{gZ}(gz)$ for all domains $Z \in \mathfrak S$, elements $g \in \Aut(\mathfrak S),$ and points $z \in X.$

 We claim that $b^{Mm}(H^+ \cup H) \subsetneq H^-$. In particular, $b^{Mm}H^+ \subsetneq H^-$ and $b^{Mm}H \cap H=\emptyset.$ To see this, observe that the set $\pi_V(H^+ \cup H)$ coarsely agrees with $\rho^U_V$, and on the other hand, the set $\pi_V(b^{Mm}(H^+ \cup H))$ is far from $\rho^U_V$. In particular, the sets $b^{Mm}(H^+ \cup H), H^+ \cup H$ are disjoint. As $X=H^+ \cup H \cup H^-,$ we deduce that $b^{Mm}(H^+ \cup H) \subsetneq H^-.$ This concludes the proof.  We remark that when $X$ is further a cube complex, since $H$ contains a cubical hyperplane $k$ with $H^- \subset k^-$ and $ k^+ \subset H \cup H^+$ (Lemma \ref{lem:hyperplane_in_curtain}), the above shows that $b^{Mm}k^+ \subsetneq k^-$.

    \end{proof}

    \begin{corollary}\label{cor:loxodromic_upper_level} Let $(X, \mathfrak S)$ be an HHS, $g_1,g_2 \in \Aut(\mathfrak S)$, $U \in \vbig(g_1),V \in \vbig(g_2)$ and $U,V \sqsubsetneq W$ with $\dist_W(\rho^U_W, \rho^U_W)>20E$. If $g_iW=W$ for $i \in \{1,2\}$, then $\exists$  $m=m(\tau_U(g_1), \tau_V(g_2), X) \in \mathbb{N}$ a curtain $h \subset W$ such that:

    \begin{enumerate}
        \item $g_1^m$ flips $h^-$ and $g_2^m$ flip $h^+$,

        \item $g_1^mg_2^m h^+ \subsetneq h^+$ and $g_1^mg_2^mh \cap h=\emptyset$. 
    \end{enumerate}
    In particular, $g_1^mg_2^m$ is loxodromic on $\calC W$ and $\tau_W(g_1^m g_2^m) \geq 1.$ 
        
    \end{corollary}

\begin{proof}
    The statement follows immediately by combining Corollary \ref{cor: flip_then_skewer} with the above Lemma \ref{lem:HHS_flip}. Namely, continuing with the notation from the above Lemma \ref{lem:HHS_flip} with $g_2=b$, we get an integer $m$ such that $g_2^m$ flips $h^+$, similarly, $g_1^m$ flips $h^-$. Hence, by Corollary \ref{cor: flip_then_skewer}, the element $g_1^mg_2^m$ skewers $h \subset \calC W$ as it satisfies $g_1^mg_2^m(h \cup h^+) \subsetneq h^+$. Hence, by Lemma \ref{lem:translation__length_via_skewering}, the element $g_1^mg_2^m$ is loxodromic on $\calC W$ and $\tau_{W}(g_1^mg_2^m)>1$.
\end{proof}

In the special case where the HHS in Corollary \ref{cor:loxodromic_upper_level} is a CAT(0) cube complex,  Lemma \ref{lem:hyperplane_in_curtain} tells us that the respective HHS curtain $H$ contains a cubical hyperplane $k.$ This yields the following.

\begin{corollary}\label{cor:flip_cube} Let $X$ be a CAT(0) cube complex with an HHS structure $(X, \mathfrak S)$, $g_1,g_2 \in \Aut(\mathfrak S)$, $U \in \vbig(g_1),V \in \vbig(g_2)$ and $U,V \sqsubsetneq W$ with $\dist_W(\rho^U_W, \rho^U_W)>20E$. If $g_iW=W$ for $i \in \{1,2\}$, then $\exists$  $m=m(\tau_U(g_1), \tau_V(g_2), X) \in \mathbb{N}$ a cubical hyperplane $k \subset X$ such that:

    \begin{enumerate}
        \item $g_1^m$ flips $k^-$ and $g_2^m$ flip $k^+$,

        \item $g_1^mg_2^m k^+ \subsetneq k^+$ and $g_1^mg_2^mk \cap k=\emptyset$. 
    \end{enumerate}
    Further, $g_1^mg_2^m$ is loxodromic on $\calC W$ and $\tau_W(g_1^m g_2^m) \geq 1.$

\end{corollary}

%In the special case where the domains $U,V$ in Corollary \ref{cor:loxodromic_upper_level} are translates of one another, one can gets a more explicit form for the loxodromic $g_1^mg_2^m$:

%\begin{corollary}\label{cor:special_case_of_loxodromic_upper_level} Let $U \in \vbig(a)$ and $U,bU \nest W$ for some $b \in \Aut(\mathfrak S)$ with $\dist_W(\rho^U_W,\rho^{bU}_W)>20E$. If $aW=W$ and $bab^{-1}W=W$, then $ \exists m=m(\tau_U(a), \mathfrak S) \in \mathbb{N}$ such that $g:=a^m(ba^mb^{-1})$ skewers a half space in $\calC W$. Furthermore, $g$ is loxodromic on $\calC W$ with $\tau_W(g) > 1.$
    
%\end{corollary}

\section{Algebraic passing up}

The main goal of this section is to start with an element $a \in \Aut({\mathfrak S})$ which is loxodromic over some $\calC U$ and use its (short) conjugates to produce a loxodromic element $h$ on some $W$ with $U \nest W.$ We start with the following.

\begin{lemma}\label{lem:effective_first_pass} (effective passing up) Let $U \in \vbig(a)$ for some $a \in \Aut(\mathfrak S)$. For each $K_1>K_2>50E$, there exists a constant $m=m(\mathfrak S,\tau_U(a), K_1)$ such that the following holds. If $bU$ is transverse to $U$ for some $b \in \Aut(\mathfrak S)$, then there exist elements $h_1,h_2 \in \langle a,b \rangle $ which are conjugates of $a$ by elements from $\langle a,b \rangle$ such that:

\begin{enumerate}
    \item $|h_i|_{a,b}<m$,
    \item $U,h_1U, h_2U$ are nested in some $W$, and 
    \item $\dist_W(\rho, \rho')>K_2$ for any $\rho, \rho' \in \{\rho^U_{W}, \rho^{h_1U}_W, \rho^{h_2U}_W\}$.
\end{enumerate}
\end{lemma}

\begin{proof}
    Set $U_0=U$ and let $U_1=bU$. Using Lemma \ref{lem:Big_Sets}, there exists a constant $M=M(\mathfrak S)$ such that $a^MU=U.$ Hence, the element $g_1:=ba^Mb^{-1}$ is active over $U_1$ and consequently, there exists a constant $s$ depending only on $\tau_U(a)$ and $K_1$ such that $\dist_{U_1}(\rho^U_{U_1}, g_1^{s}\rho^U_{U_1})=\dist_{U_1}(\rho^U_{U_1}, \rho^{g_1^{s}U}_{U_1})>K_1$. This provides domains $U_0=U,U_1, U_2=g_1^{s}U$. Now, observe that the element $g_2=g_1^{s}a^{M}g_1^{-s}$ is active over $U_2.$ Thus, again, for the same constant $s$ above, we get that $\dist_{U_2}(\rho^{U_1}_{U_2}, g_2^{s}\rho^{U_1}_{U_2})=\dist_{U_2}(\rho^{U_1}_{U_2}, \rho^{g_2^{s}U_1}_{U_2})>K_1.$ Repeating this argument $P$-times  (where $P$ depends only on $K_1$ as in Lemma \ref{lem:strong_passing}) provides domains $\{U_0, U_1, \cdots U_P\}$ such that $\dist_{U_i}(\rho^{U_{i-1}}_{U_i}, \rho^{g_i^sU_{i-1}}_{U_i})>K_1$ for each $i \in \{1, \cdots P\}.$ Since the collection $\calU= \{U_0,U_1,\cdots U_P\}$ consists of translates of $U$, applying Lemma \ref{lem:strong_passing_up_consequence} to $\calU$ gives the desired conclusion.
\end{proof}

The above Lemma \ref{lem:effective_first_pass} allows us to pass up to a higher domain $W$ where conditions 1,2 and 3 are met. If the elements $h_1$ and $h_2$ fix $W$, we can apply Lemma 5.7 again replacing $a$ by $h_1^mh_2^m$ from Corollary \ref{cor:loxodromic_upper_level} (and then we can iteratively apply Lemma 5.7 until we arrive at the highest domain possible and get a loxodromic there, again by applying Corollary \ref{cor:loxodromic_upper_level}). However, it need not be the case that both $h_i$ fix $W$. The following lemma is meant to handle such a situation. The lemma will essentially tell us that such a case is non-existent in the sense that if $h_iW=W$ for $i=1$ and $i=2,$ then we can pass to a higher domains $Z$ with $W \sqsubsetneq Z$ such that the original domains from the hypothesis are still far; and hence, applying the lemma finitely many times assures the existence of a domain $Z'$ where both $h_i$ fix $Z'$. We state and prove the lemma in a slightly more general form than needed in this article as it may find other applications.

\begin{lemma}\label{lem:effective_second_pass}  For $i \in \{1,2\}$, let $U_i \in \vbig(a_i)$ for some $a_i \in \Aut(\mathfrak S)$ and let $K_1>K_2>50E$. Suppose that $\dist_W(\rho^{U_1}_W,\rho^{U_2}_W)>K_1$ for some $W$ with $U_i \nest W$. If $a_iW \neq W$ for each $i \in \{1,2\}$ then there exist $m=m(\tau_{U_1}(a_1), \tau_{U_2}(a_2), \mathfrak S, K_1)$, a domain $Z \in \mathfrak S$ with $W \sqsubsetneq Z$ and elements $h_1,h_2 \in \langle a_1,a_2\rangle $ such that:

\begin{enumerate}
    \item $|h_j|_{a_1,a_2}<m$ for $j \in \{1,2\},$
    \item $\dist_Z(\rho_Z,\rho'_Z)>K_2$ for each distinct $\rho_Z, \rho'_Z \in \{\rho^{W}_Z, \rho^{h_1W}_Z,\rho^{h_2W}_Z\}$.

\end{enumerate}

In particular, if $i \in \{1,2\}$,  we have $\dist_Z(\rho_Z, \rho'_Z)>K_2$ for any $\rho_Z, \rho_Z' \in \{\rho^{U_i}_Z, \rho^{h_1U_i}_Z, \rho^{h_2U_i}_Z\}$. 

\end{lemma}

\begin{proof}
To simplify notation, let $a:=a_1, b:=a_2$ and $U:=U_1, V:=U_2.$ Using Lemma 6.3 in \cite{Durham2017-ce}, we get that there is a constant $M=M(\mathfrak S)$ such that $a^MU=U$ and $b^MV=V.$ For simplicity, we will continue the proof writing $a$ instead of $a^M$ and similarly for $b.$

\begin{figure}[ht]
   \includegraphics[width=\textwidth, trim = 2cm 10cm 1cm 7cm]{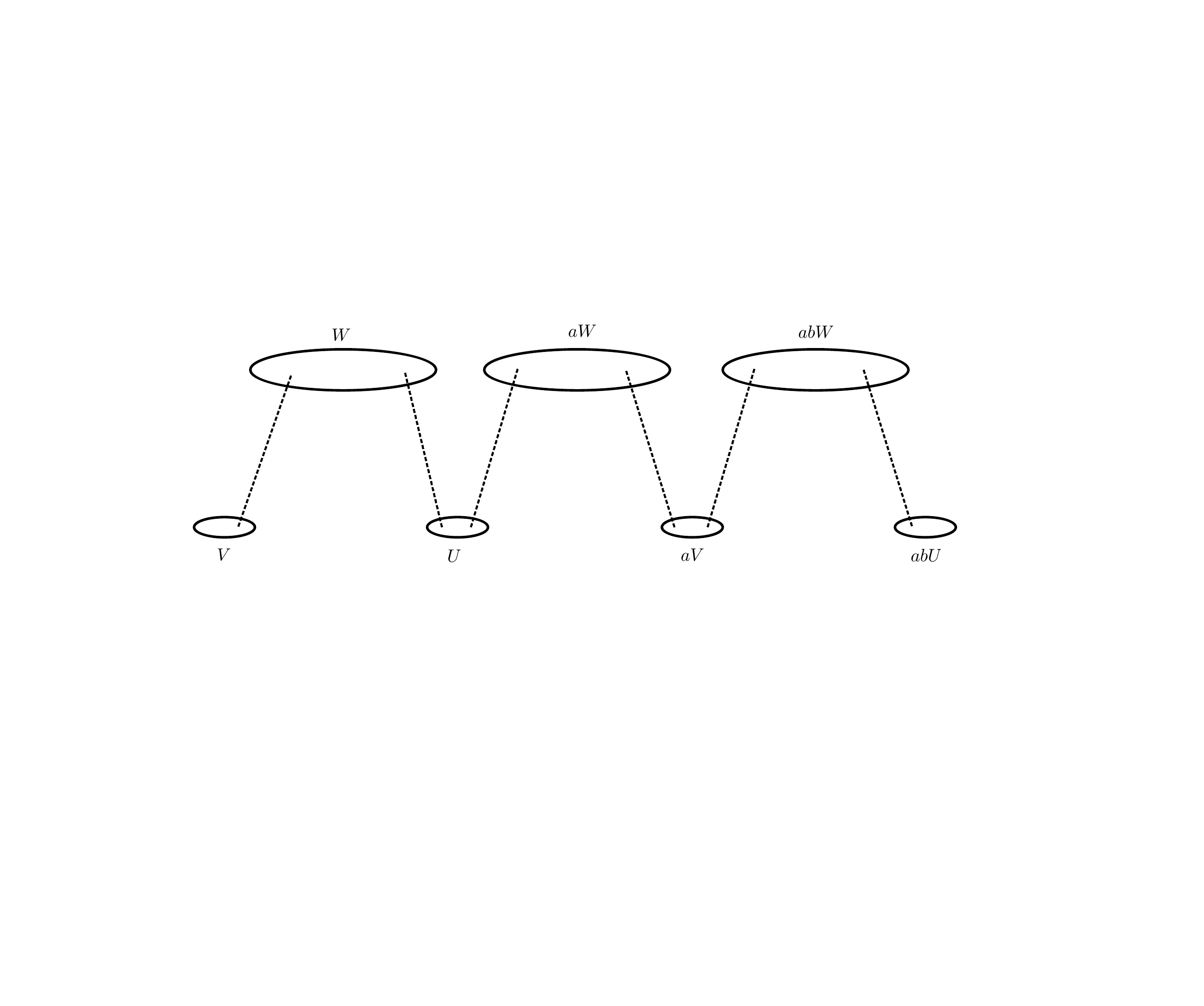}\centering
\caption{Producing transverse domains from $W$} \label{fig:nice_nest}
\end{figure}
\vspace{2mm}

Observe that every two consecutive domains of the tuple $$\calW= (W,aW,abW,abaW,ababW, \cdots )$$ are distinct as otherwise we would have $aW=W$ or $bW=W$ violating the assumptions of the lemma. Let $g_i$ be such that the $i$-th element of $\calW$ is $g_iW.$

\begin{claim*}
    Every two consecutive terms of $\calW$ are transverse.
\end{claim*}

\begin{claimproof} We start by showing $W,aW$ are transverse. Apply the element $a$ on the triple $\{W,U,V\}$ to get the triple $$\{aW, aU, aV\}=\{aW, U,aV\},$$ where the equality holds as $aU=a.$ The domains $W,aW$ can't be orthogonal since $U \nest W$ and $U \nest aW.$ They also can't be nested since the action preserves $\nest$-levels, thus, they are transverse. Furthermore, since $U$ nests in both $W,aW$, we have $\dist_W(\rho^U_W, \rho^{aW}_W)<E, \dist_{aW}(\rho^U_{aW}, \rho^W_{aW})<E$ (using th 4-th item of the definition of an HHS in Section \ref{sec:HHS_prelim}). Finally, since $a:W \rightarrow aW$ is an isometry and since $a\rho^Z_{Z'}=\rho^{aZ}_{aZ'}$ for any $Z,Z' \in \mathfrak S,$ we have $\dist_{aW}(\rho^{U}_{aW}, \rho^{aV}_{aW})=\dist_W(\rho^U_W, \rho^V_W).$

Now, applying $aba^{-1}$ to the new collection $\{aW,U,aV\}$ provides the collection $$\{abW,aba^{-1}U,abV\}=\{abW,abU,aV\},$$ where the equality holds since $bV=V$ and $a^{-1}U=U.$ Since $aV$ nests in both $\{aW,abW\}$, arguing exactly as above gives us that $aW,abW$ are transverse and $\dist_{aW}(\rho^W_{aW}, \rho^{abW}_{aW})>K_1$. Continuing inductively proves the desired claim.

\end{claimproof}

Therefore, invoking Corollary \ref{cor:local-to-global} gives us that $\calW$ is a collection of distinct pairwise transverse domains. Using Lemma \ref{lem:strong_passing}, we get a domain $J,$ and three domains $g_iW$ with $g_iW \nest J$ for each $i \in \{1,2,3\}$ and $\dist_J(\rho^{g_iW}_J, \rho^{g_jW}_J) >K_2$ for each $i \neq j.$ Applying $g_1^{-1}$ on the collection $\{J,g_1W,g_2W,g_3W\}$ and defining $Z:=g_1^{-1}J,$ $h_1:=g_1^{-1}g_2$ and $h_2:=g_1^{-1}g_3$ gives us the desired statement.
\end{proof}

%Although we have attempted to state Lemma \ref{lem:effective_second_pass} above in the most generality possible, in practise, we only need it in the case where the domain $U_2$ is a translate of $U_1.$ In this special case, one can deduce that the resulting elements $h_1,h_2$ are conjugates of $a_1.$ Namely, we have the following.

%\begin{corollary}\label{cor:second_pass}  Let $U,bU \in \mathfrak S$ with $U \in \vbig(a)$ and let $K_1>K_2>50E$. Suppose that $\dist_W(\rho^{U}_W,\rho^{bU}_W)>K_1$ for some $W$ with $U,bU \nest W$. If $aW \neq W$ and $bab^{-1}W \neq W$, then there exist $m=m(\tau_{U_1}(a), \mathfrak S, K_1)$, a domain $V \in \mathfrak S$ with $W \nest V$ and elements $h_1,h_2\in <a,b>$ that are conjugates of $a$ such that:

%\begin{enumerate}
 %   \item $|h_i|_{a,b}<m$ for $i \in \{1,2\},$
  %  \item $\dist_V(\rho,\rho')>K_2$ for each distinct $\rho, \rho' \in \{\rho^{W}_V, \rho^{h_1W}_V,\rho^{h_2W}_V\}$.

  %  \end{enumerate}
%\end{corollary}

%\begin{proof}
 %   This follows immediately by the previous Lemma \ref{lem:effective_second_pass} by declaring $a_1=a,$ $a_2=bab^{-1}$. Further, the constant $m$ here only depends $K_1, \mathfrak S$ and the translation length of $a$ since $a_1,a_2$ are conjugates and thus of the same translation length.
%\end{proof}
%\begin{lemma}\label{lem:final_pass} There exists $M=M(X)$ such that if $U \in Big(a)$ and $U,bU$ are transverse for some $b$, then there exist elements $g_i$ and domains $W_i \in Big(g_i)$ such that $|g_i|_{a,b}<M$ and $\dist_S(\rho^{W_1}_S, \rho^{W_2}_S)>50E.$
%\end{lemma}

The goal of the following lemma is to start with a pair of transverse domains and iteratively apply the passing-up-type statements we have already established to arrive to the highest domain possible and produce a loxodromic on that domain by Corollary \ref{cor:loxodromic_upper_level}.

\begin{lemma}\label{lem:short_loxodromic_top}  Let $T$ be a finite generating set of a subgroup $H$ with eyries $\{S_1,S_2,\cdots S_M\}$ and assume $H.S_i=S_i$ for each $i.$ Suppose $a \in H$,  $U \in \vbig(a)$ and $U \sqsubsetneq S_j$ for some $j\in \{1,\cdots M\}$. There exist $m=m(\tau_U(a), \mathfrak S) \in \mathbb{N}$ and $g \in H$ such that:   
     \begin{enumerate}

     \item $|g|_T<m$, and
        \item $\dist_{S_j}(\rho^{U}_{S_j}, \rho^{gU}_{S_j})>30E.$
    \end{enumerate}

    In particular, the element $h=a^m(g^{-1}a^mg)$ is loxodromic on $S_j$ with $\tau_{S_j}(h) \geq 1.$
\end{lemma}

\begin{proof} For simplicity, let's denote $S=S_j$. By Lemma \ref{lem:producing_first_transverse_pair}, the collection $\cup_{i=0}^{N+1}\{T^iU\}$ contains a pair of transverse domains; and after to translating, we may assume that this pair is of the form $\{U,bU\}$ where $b \in T \cup \cdots \cup T^{N+1}$. As $U \nest S$ and $H.S=S,$ we have $bU \nest bS=S.$ Fix $K_1>K_2>50E$, using Lemma \ref{lem:effective_first_pass}, we get a constant $m=m(\mathfrak S, \tau_U(a))$, a domain $W,$ and elements $h_1, h_2 \in \langle a,b \rangle $ such that $|h_i|_{a,b}<m$, $U,h_1U,h_2U \nest W$, and $\dist_W(\rho, \rho')>K_2$ for any distinct $\rho, \rho' \in \{\rho^U_{W}, \rho^{h_1U}_W, \rho^{h_2U}_W\}$. %Further, each $h_i$ is a conjugate of $a$ by an element of $\langle a,b \rangle$.

There are exactly two cases to consider:

    \begin{itemize}
        \item Case 1: There is a pair $x,y \in \{a,h_1ah_1^{-1},h_2ah_2^{-1}\}$ with $xW=W$ and $yW=W$.

        \item Case 2: There is a pair $x,y \in \{a,h_1ah_1^{-1},h_2ah_2^{-1}\}$ with $xW \neq W$ and $yW \neq W$.

    \end{itemize}
Before we proceed with the proof, observe that Cases 1 and 2 indeed exhaust all possible scenarios. If Case 1 occurs, then there exists an $n=n(\tau_U(a), \mathfrak S) \in \mathbb{N}$ such that $(x^n.y^n)W=W$ and $x^ny^n$ is a  loxodromic on $\calC W$ with $\tau_W(x^ny^n)>1$ by Corollary \ref{cor:loxodromic_upper_level}. Now we are exactly in the same position as the assumptions of this lemma (with $U=W$, $a=x^ny^n$) but with $W$ being higher in the nesting than $U$ is. Hence, if we apply exactly the same argument again and keep landing in Case 1, the penultimate step produces an element $g \in \vbig(J)$, domains $g_1 J, g_2J$ with $|g_i|_T$ bounded above depending only on $\tau_U(g)$ (which is at least $\text{min}\{\tau_U(a), 1\}$), and $\dist_S(\rho, \rho')>K_2$ for any distinct $\rho, \rho' \in \{\rho^J_S, \rho^{g_1J}_S, \rho^{g_2J}_S\}$. Furthermore, since $U \nest W \nest \cdots \nest J$, we have $U \nest J,$ and in particular, we have $g_iU \nest g_iJ$. Therefore, up increasing the increasing constant $K_2$ to $K_2+E,$ we have $\dist_S(\rho, \rho') \geq K_2$ for any $\rho, \rho' \in \{\rho^U_S, \rho^{g_1U}_S, \rho^{g_2U}_S\}$. Now, we can invoke Corollary \ref{cor:loxodromic_upper_level} with the domains $U,g_1U$ to get that the element $a^m(g_1a^mg_1^{-1})$ is loxodromic on $S.$ If Case 2 occurs, then by applying Lemma \ref{lem:effective_second_pass} at most $N$-times (where $N$ is the complexity of the HHS), we get a domain $Z$ where Case 1 occurs, this concludes the proof.

%(importantly, the translation length of the new element $x^ny^n$ over $W$ is bounded below by 1. Also, since $|h_i|_{a,b}<m$ and $|b|_T$ is uniformly bounded above by $|T|+1$, the length $|x^ny^n|_{a,b}$ depends only on $m,n$ and both of which are determined by $\tau_U(a)$).

%Suppose Case 2 holds instead, then, we can invoke Lemma \ref{lem:effective_second_pass} (with $a_1=x, a_2=y$) to get a domain $V$ with $W \nest V$ and elements $\{h_j'\}_{j=1}^2$ where
  % $\dist_V(\rho,\rho')>K_2$ for each distinct $\rho, \rho' \in \{\rho^{W}_V, \rho^{h_1'W}_V,\rho^{h_2'W}_V\}$. Since $U \nest W$ and $W,h_j'W \nest V$, we get that $U,h_j'U \nest V$ for $j=1,2.$ Further, since $h_j'U \nest h_j'W \nest V$, we have $\dist_V(\rho,\rho')>K_2$ for any distinct $\rho,\rho' \in \{\rho^U_V, \rho^{h_1'U}_V, \rho^{h_2'U}_V\}$ (observe that since $x,y$ are conjugates of $a,$ the statement of Lemma \ref{lem:effective_second_pass} gives us that $h_j' \in \langle x,y \rangle$ and $|h_j'|_{x,y}$ depends only on $\tau_U(a)$). Now, again, we must land in precisely in Case 1 or Case 2 above but with $V$ instead of $W$ where $V$ is higher in the nesting than $W$. Repeating the above at most $N$-many times concludes the proof.

 \end{proof}

%\begin{lemma}\label{lem:Effective_rank_rigidity}
  %  Let $T$ be a finite generating set of a subgroup $H$ with eyries $\{S_1,S_2,\cdots S_n\}$ and assume $H.S_i=S_i$ for all $i$. If $U \in \vbig(a)$ for some $a \in H$ then there exist some $m=m(\tau_U(a), \mathfrak S)$ and $h \in H$ with $|h|_T$ depending only on $m$ and $|a|_T$ such that $w=a^m(ha^mh^{-1})$ is loxodromic on $S_i$, where $U \sqsubseteq S_i.$ Finally, $\tau_{S_j}(w)>1.$
%\end{lemma}

%\begin{proof} Let $U,a$ be as in the statement. By Theorem \ref{thm:eyries}, the domain $U \sqsubseteq S_i$ for some $i.$ If $U=S_i,$ there is nothing to prove (as $h$ can be taken to be $e$ in this case and $m$ is taken large enough, depending only on $\tau_{S_i}(a)$, so that $\tau(a^m)>1$) so assume that $U \nest S_i.$ By Lemma \ref{lem:producing_first_transverse_pair}, there is an element $b \in T \cup T^2 \cdots T^{N+1}$ such that $U,bU$ are transverse. Now, since $|b|_T<N+1,$ Lemma \ref{lem:short_loxodromic_top} yields an integer $m=m(\tau_U(a),\mathfrak S)$ and an element $h$ such that $a^m(ha^mh^{-1})$ is loxodromic on $S_i$, and $|h|_T$ depends only on $m$ and $|a|_T$, as $|b|_T \leq N+1.$
%\end{proof}

For $U \in \vbig(a)$, the constants in the previous statements all heavily depended on $\tau_U(a)$. The following lemma states that one can always find some $U \in \vbig(a)$ where $\tau_U(a)$ is uniformly bounded below.

\begin{lemma}[{\cite[Theorem 1.5 ]{Abbott-Hagen-Petyt-Zalloum23}}]\label{lem:lower_bound_translation} Let $G$ be an HHG given with a structure $(G, \mathfrak S)$. There exists some $A=A(\mathfrak S)>0$ such that for any non-torsion $a \in G,$ there exists some $U \in \vbig(a)$ such that $\tau_U(a^M)>A,$ where $M$ depends only on the complexity of $\mathfrak S.$    
\end{lemma}

Let $A$ as in Lemma \ref{lem:lower_bound_translation}. For each non-torsion element $a \in G,$ we will denote $\vvbig(a):=\{U \in \mathfrak S| \,\, \tau_U(a^M)>A\}.$ The above Lemma \ref{lem:lower_bound_translation} states that $\vvbig(a) \neq \emptyset$ for all non-torsion $a \in G.$ In particular, if $U$ in Lemma \ref{lem:short_loxodromic_top} is chosen in $\vvbig(a)$ and $a$ is a generator, the resulting constant $m$ will be uniform:

\begin{corollary}\label{cor:Effective_rank_rigidity_torsion_free} 
    Let $H$ be a virtually torsion-free subgroup with eyries $\{S_1,S_2,\cdots S_n\}$, there exists $m=m(\mathfrak S)$ such that for any finite generating set $T$ of $H$, there exists $h \in H$ with $|h|_T<m$ such that $h$ is loxodromic on some $S_i$ with $\tau_{S_i}(h)>min\{\frac{A}{M},1\}.$
\end{corollary}

\begin{proof}
    Let $H_1<H$ be a finite-index torsion-free subgroup. If we pass to a finite index subgroup once more, we get a subgroup $H_2 <H_1$ such that that $HS_i=S_i$ for all $i.$ It's immediate to check that the eyries provided by Theorem \ref{thm:eyries} for $H_2$ are exactly the same eyries for $H$, given by $\{S_1,\cdots S_n\}.$ Let $T_2$ be the finite generating set provided by Lemma \ref{lem:pass_to_finite_index}, in particular, the $T$-length of each element in $T_2$ is bounded above by $2d-1$ where $d=[H:H_2].$ Let $a \in T_2$ and $U \in \vvbig(a)$ (which is possible by Lemma \ref{lem:lower_bound_translation}), in particular, $\tau_U(a^{M})>A$ for some $A=A(\mathfrak S)$ and $M=M(\mathfrak S)$. Since $M$ is uniform and depends only on the complexity of $\mathfrak S$, we will abuse notation and use $a$ to denote $a^{M}$. If $U=S_i$ for some $i$, there is nothing to prove (and this the case where $\tau_S(a) \geq \frac{A}{M}$). Otherwise, by Theorem \ref{thm:eyries}, $U \nest S_i$ for some $i,$ and hence, applying Lemma \ref{lem:short_loxodromic_top} to $a$ provides an integer $m=m(\mathfrak S,A)=m(\mathfrak S)$, an element $g \in H_2$ with $|g|_{T_2}<m$ such that $h=a^m.(ga^mg^{-1})$ is loxodromic on $S_i$ with $\tau_{S_i}(h)>1$. But $m$ depends only on $\mathfrak S$, and $|a|_{T}$ is bounded above by $2d-1$ where $d=d(H)$, concluding the proof.
\end{proof}

\section{Consequences: main results}

In this section, we conclude the main results of this article.

\subsection{Effective rank-rigidity}

Combining Corollary \ref{cor:Effective_rank_rigidity_torsion_free} with the rank-rigidity statement in \cite{Durham2017-ce} (see also \cite{PetytSpriano20}) gives us the following, which is the main result of this article.

\begin{corollary} \label{cor:effective_rank_regidity_HHG}(effective rank-rigidity) Let $G$ be a virtually torsion-free HHG. There exists a constant $m=m(G)$ such that for any finite generating set $T$ for $G$, precisely one of the following holds:

\begin{enumerate}
    \item There exists a Morse element $g \in G$ with $|g|_T<m$, or

    \item The group $G$ is quasi-isometric to a product of two unbounded HHses.
\end{enumerate}
Moreover, when $(1)$ occurs, $G$ contains a free all-Morse stable subgroup $\langle g_1,g_2 \rangle $ with $|g_i|_T<m$ unless $G$ is virtually cyclic.
    
\end{corollary}

\begin{proof} Using Corollary 4.7 in \cite{PetytSpriano20}  (or Theorem 9.14 in \cite{Durham2017-ce}), if $G$ is not quasi-isometric to a product of unbounded HHSes, then it contains a Morse element $w$. In particular, using Corollary \ref{cor:singlton}, $G$ has precisely one eyry, call it $S.$ The statement is then immediate by Corollary \ref{cor:Effective_rank_rigidity_torsion_free} concluding the proof of the dichotomy. Notice that the provided $g$ in case 1 is not merely Morse, but it also acts loxodromically on $\calC S$ (and the two are not necessarily equivalent for a general HHS structure \cite{ABD}). For the moreover part, first, let $g$ be the Morse element provided by part 1, as mentioned above, such a $g$ acts loxodromically on $\calC S.$ By Theorem K in \cite{HHS1} the $G$-action on $\calC S$ is acylindrical and hence, by Lemma 6.5 in \cite{Dahmani2017}, if every generator stabilizes $\{g^+,g^-\} \in \partial \calC S,$ then $G$ is virtually cyclic. Thus, by Corollary 6.6 of \cite{Dahmani2017}, there exists $b \in T$ with $bg^kb^{-1} \neq g^{\pm k}$ for all $k \neq 0.$ Applying Theorem \ref{thm:fuj0809}, we conclude that there is a constant $n$ such that for all $m \geq n$, the group $H=\langle g^m, bg^mb^{-1} \rangle $ is a free subgroup which is quasi-isometrically embedded in $\calC S.$ To show $H$ is stable, we can argue in a few different ways, for instance, the group $G$ admits a largest aclyndrical action on the hyperbolic space constructed in \cite{ABD}, and since $H$ quasi-isometrically embeds in $\calC S,$ it must also quasi-isometrically embed in $\calC S'.$ Thus, $H$ is stable by Theorem B in \cite{ABD}. Since $H$ is free (in particular, torsion-free), each element $h \in H$ acts loxodromically on both $\calC S, \calC S'$ and hence $H$ is all-Morse (either by Corollary \ref{cor:singlton} or Theorem B in \cite{ABD}).

\end{proof}

\begin{corollary}\label{cor:general_HHG_body} Let $(G, \mathfrak S)$ be an irreducible HHG. There exists an integer $m=m(\mathfrak S)$ such that for any finite generating set and any infinite infinite order element $a \in G$, there is an element $g \in G$ with $|g|_T<m$ and $a^m(ga^mg^{-1})$ is Morse.
    
\end{corollary}

\begin{proof}
Since $G$ is irreducible, it has exactly one eyry call it $S.$ For an infinite order element $a \in G,$ if $a \in \vbig(S)$, then we take $g=e.$ Otherwise, by Lemma \ref{lem:lower_bound_translation}, there exists an element $U \sqsubsetneq S$ with $U \in \vvbig(a)$ and by Lemma \ref{lem:short_loxodromic_top}, there exists an integer $m=m(\tau_U(a), \mathfrak S)=m(A,\mathfrak S)$ and an element $g \in G$ with $|g|_T<m$ where $a^m(ga^mg^{-1})$ is loxodromic on $S$ and hence Morse.
\end{proof}

Since groups acting geometrically on CAT(0) cube complexes with factor systems are HHGs, using work of Caprace and Sageev \cite{capracesageev:rank}, we obtain the following.

\begin{corollary}(effective rank-rigidity for cubulated groups with factor systems)\label{cor: effective_rank_CCC} Each CAT(0) cube complex with a factor system $X$ determines a constant $m=m(X)$ such that for any group $G$ acting freely cocompactly on $X$ and any finite generating set $T$ for $G,$ precisely one of the following occurs:

\begin{enumerate}
    \item $X=X_1 \times X_2$ where each $X_i$ is an unbounded CAT(0) cube complex, or
    \item $\mathrm{Cay}(G,T)$ contains a rank-one element in the ball of radius $m$ centered at the identity.

\end{enumerate}
Moreover, when $(2)$ occurs, $G$ contains a free-all-rank-one stable subgroup $\langle g_1,g_2 \rangle $ with $|g_i|_T<m$ unless $G$ is virtually cyclic.
    
\end{corollary}

\begin{proof}
    This is an immediate consequence of Corollary \ref{cor:effective_rank_regidity_HHG}, the rank-rigidity statement in \cite{capracesageev:rank}. The only part that we need to explain is the claim that $m$ depends only on $X$ and not on $G;$ this is explained as follows.

 If $X$ is a CAT(0) cube complex with a factor system, by \cite{HHS1}, \cite{HHS2}, the space $X$ admits an HHS structure $(X, \mathfrak S)$ with an HHS constant $E.$ Further, if $G$ is a group with a free cocompact action on $X$ (by cubical isometries), then $G$ is an HHG with respect to the fixed HHS structure $(X, \mathfrak S).$ Given an arbitrary HHG with an HHS structure $(X, \mathfrak S)$, the constant $A$ of Lemma \ref{lem:lower_bound_translation} depends only on $\mathfrak S$ \emph{and} on $\tau_0=\text{inf}\{\tau_X(g)| g \in G \text{ has infinite order}\}$; that is, $A=A(\tau_0, \mathfrak S).$ In Corollary 1.3 of \cite{Abbott-Hagen-Petyt-Zalloum23}, we show that for any HHG, the constant $\tau_0$ is positive (but its values can get arbitrarily close to zero when the group $G$ changes). In general, $\tau_0$ will heavily depend on the properness of the $G$-action on the HHS $(X, \mathfrak S).$ However, when $X$ is a CAT(0) cube complex, every infinite order cubical isometry $g$ of $X$ admits a combinatorial geodesic axis \cite{haglund:isometries}, and hence, $\tau_0 \geq 1$. In particular, $A=A(\tau_0, \mathfrak S)=A(\mathfrak S),$ and the HHS structure $\mathfrak S$ is fully determined by the CAT(0) cube complex $X,$ latter depends only on the cube complex, this concludes the proof.

\end{proof}

Combining the previous discussion with Corollary \ref{cor:general_HHG_body} gives us the following.

\begin{corollary} Let $X$ be an irreducible CAT(0) cube complex with a factor system. There exists $m=m(X)$ such that for any group acting properly cocompactly by on $X$, if $a \in G$ has infinite order, then there exists an element $g \in G$ with $|g|_T<m$ such that $a^m.(ga^mg^{-1})$ is rank-one.
    
\end{corollary}

%Let $X$ be a CAT(0) cube complex with a factor system, by \cite{HHS1}, \cite{HHS2}, the space $X$ admits an HHS structure $(X, \mathfrak S)$ with an HHS constant $E.$ If $G$ is a group with a free cocompact action on $X$ (by cubical isometries), then $G$ is an HHG with respect to the fixed HHS structure $(X, \mathfrak S).$ Given an arbitrary HHG with an HHS structure $(X, \mathfrak S)$, the constant $A$ of Lemma \ref{lem:lower_bound_translation} depends only on $\mathfrak S$ and on $\tau_0=\text{inf}\{\tau_X(g)| g \in G \text{ has infinite order}\}$. In Corollary 1.3 of \cite{Abbott-Hagen-Petyt-Zalloum23}, we show that for any HHG, the constant $\tau_0$ is positive. In general, $\tau_0$ will heavily depends on the properness of the $G$-action on the HHS $(X, \mathfrak S).$ However, when $X$ is a CAT(0) cube complex, every infinite order cubical isometry $g$ of $X$ admits a combinatorial geodesic axis, and hence, $\tau_0 \geq 1$. Finally, the constant $M$ in Lemma \ref{lem:lower_bound_translation} depends only on the complexity of the HHS and not on the group acting. Therefore, the constant $A$ in Lemma \ref{lem:lower_bound_translation} is determined by $\mathfrak S$ and $M$ both of which are parameters of $X$ only, this concludes the proof.

\begin{remark}\label{rmk:essential_contact_graph}
    In the special case where $X$ in the above Corollary \ref{cor: effective_rank_CCC} is essential (which can always by assured under a geometric action; by replacing $X$ with its essential core \cite{capracesageev:rank}), the elements $g$ (and $g_1, g_2$) acts loxodromically on the contact graph. This is because the element $g$ we produce in the proof of Corollary \ref{cor:effective_rank_regidity_HHG} acts loxodromically on the unique eyry for $G$, but that can be taken to be the contact graph when $X$ is essential and irreducible (for instance, see the discussion at the beginning Appendix A in \cite{Abbot-Ng-Spriano}).
\end{remark}

%\begin{remark}(HHS structures) For more details on such structures, see \cite{HHS1} particularly Remark 13.2, the discussion leading to Proposition A.6 in \cite{Abbot-Ng-Spriano} and \cite{ABD}). First, we recall that a CAT(0) cube complex $X$ is said to be \emph{essential} if for each hyperplane $h,$ both $h^+$ and $h^-$ contain points from $X$ that are arbitartily far from $h$ (more generally, a half space $h^+$ with this property is said to be \emph{deep}). There are two standard HHS sructures such an $X$ can be equipped with, we are only interested in the geometry the maximal domains in such structures.

%We fix an irreducible CAT(0) cube complex $X$ with a factor system on which $G$ acts properly cocompactly. 

%\begin{enumerate}
 %   \item The first HHS structure structure has the contact graph $\calC X$ as its unique maximal domain, and when $X$ is essential, $\calC X$ is unbounded.

  %  \item The second
%\end{enumerate}

%\end{remark}

\begin{remark}(\cite{ABD})\label{rmk:ABD} Every HHG admits an HHS structure $(X, \mathfrak S)$ such that the following is true for the unique maximal domain $S \in \mathfrak S$:

\begin{enumerate}
    \item The hyperbolic space associated to $S$ is a point if and only if $X$ is quasi-isometric to a product.

    \item $G$ admits a (largest) acylindrical action on the $\calC S$ provided $\calC S$ is unbounded. In fact, the eyries of $G$ consists of a singlton if and only if $\calC S$ is unbounded. In this case, the unique eyry is exactly $S.$
    \item If $G$ is not virtually cyclic and the unique maximal domain is not a point, then it's unbounded and not quasi-isometric to a line. This particular fact can also be seen by Theorem \ref{thm:Characterization_Abelian}. Such a domain also agrees with the eyry for $G.$

    \item If $X$ is an irreducible unbounded CAT(0) cube complex, then $\calC S$ is unbounded. Namely, using rank-rigidity of CAT(0) cube complexes \cite{capracesageev:rank}, if $X$ is irreducible, then $G$ contains a rank-one element and hence the eyry of $G$ is a singleton (by Corollary \ref{cor:singlton}) which is exactly $S.$
\end{enumerate}
    
\end{remark}

\begin{corollary}(Strongly effective Tits alternative) Let $G$ be virtually special compact group. There exists some $R=R(G)$ such that for any finite generating set $T$ for $G$, precisely one of the following happens:

\begin{enumerate}
     
\item $G$ is virtually Abelian.
\item $G$ contains a free all-rank-one subgroup $\langle g_1,g_2 \rangle$ with $|g_i|_T<m,$ or

\item $G$ contains a free subgroup $\langle g_1,g_2 \rangle$ with $|g_i|_T<m$ and has empty Morse boundary.

%\item $X$ is a product of unbounded, irreducible CAT(0) cube complexes $X_i$ where at least one $X_j$ is not a quasi-line. In this case, $G$ contains a pair $h_1,h_2$ generating a free subgroup $\langle h_1,h_2 \rangle$ all of whose elements are rank-one on $X_j$, where $|h_1|_T,|h_2|_T<m.$

\end{enumerate}
\end{corollary}

\begin{proof} We may assume that $G$ is infinite since otherwise it's virtually Abelian. We argue similar to Theorem A.1 in \cite{Abbot-Ng-Spriano}. By definition of $G$, a finite index subgroup $G_1<G$ acts freely cocompactly cospecially (meaning $X/G_1$ is compact and special) on a CAT(0) cube complex $X$. If $X$ is irreducible, then case 2 occurs by Corollary \ref{cor: effective_rank_CCC}. By Proposition 2.6 in \cite{capracesageev:rank}, $X$ admits a decomposition of irreducible CAT(0) cube complexes $X=X_1 \times \cdots X_n$. If each $X_i$ is bounded or a quasi-line, then $G_1$ is virtually Abelian, so assume $X_j$ is unbounded and not a quasi-line for some $j.$ By Lemma \ref{lem:special}, a finite index subgroup $G_2 \leq G_1$ splits as $G_2=H_1 \times \cdots H_k$ with $H_i$ acting geometrically on $X_i$. Up to replacing $H_i$ by a finite index subgroup, we may assume that each $H_i$ is torsion-free. Consider the pair $(H_j,X_j)$ with $X_j$ unbounded and not quasi-isometric to a line. Since $H_j$ acts geometrically on $X_j,$ the group $H_j$ is not virtually cyclic. Consider the HHG given by $H_j$, by Remark \ref{rmk:ABD}, the hyperbolic space associated to the maximal domain of $H_j$ (equivalently, the unique eyry $S$) is unbounded and not a quasi-line. Further, the image of $T$ under the natural surjection $G_2 \rightarrow H_j$ is a torsion-free finite generating set for $H_j$ as $H_j$ is trosion free. Since $X_j$ is irreducible, the statement follows by the ``moreover" part of Corollary \ref{cor: effective_rank_CCC}.

%Let $T'$ be the finite generating set resulting from $T$ via Lemma \ref{lem:pass_to_finite_index} and observe that since $H$ surjects onto each $H_i$, the image of $T'$ denoted $T'_i$ under this surjection generates $H_i$. Since each $X_i$ is irreducible, the eyry of each $H_i$ is a singlton $\{S_i\}.$ By the previous lemma, there are integers $m_i=m(G)$ and elements $g_i$ acting loxodromically on each $\cal S_i$ with $|g_i|<m_i.$ Now, Theorem 5.1 in \cite{ClayCalgar2018} (see also Proposition 6.68 in \cite{genevois:hyperbolicities}) provides powers $a_i$, depending only on $g_i,$ such that $g=g_1^{a_1} \cdots g_n^{a_n}$ is loxodromic on each $S_i$, and hence, $g$ is rank-one on each $X_i$.
\end{proof}

\begin{corollary}(effective omnibus)\label{cor:omnibus} If $G$ virtually acts properly, cocompactly  and cospecially  on a finite product of unbounded irreducible CAT(0) cube complexes $X=\prod X_i$, then there exist  $m=m(G)$ and $g \in G$ such that $g$ is rank-one on each $X_i$ and loxodromic on each curtain model for $X_i$ with $|g|_T<m$ for any finite generating set $T$ for $G.$ 
    
\end{corollary}

\begin{proof}  

Again, after we pass to a finite-index subgroup, we may assume $G$ is torsion free. Continuing as in the proof of the above corollary, $G$ decomposes into $G_1 \times \cdots \times G_n$ with each $G_i$ acting geometrically on $X_i$. By \cite{HHS1}, \cite{HHS2} and Remark \ref{rmk:ABD}, each $X_i$ admits an ABD HHS structure $(X_i,\mathfrak S_i)$ whose maximal hyperbolic space domain $\calC S_i$ is unbounded. The image of a generating set $T$ under the natural surjection $G \rightarrow G_i$ is a generating set for $G_i.$ Since each $X_i$ is irreducible and unbounded, by Corollary \ref{cor: effective_rank_CCC} and Remark \ref{rmk:ABD}, we get integers $m_i=m_i(G)$ and rank-one elements $g_i$ acting loxodromically on the maximal hyperbolic space $\calC S_i$ with $|g_i|<m_i.$ Now, since the element $h_i=(e,\cdots,g_i, \cdots ,e) \in G$ is rank-one on $X_i$ and trivial elsewhere, we get that the element $h=h_1 \cdots h_n \in G$ is an element which is rank-one on each $X_i$ with $|h|_T \leq m.n$, where $m=\text{max}\{m_i| i \in \{1,\cdots n\}\}$, concluding the proof.

\end{proof}

\subsection{Effective flipping and skewering}

In light of the statements we have already established, we obtain the following.

\begin{corollary} Each irreducible CAT(0) cube complex with a factor system $X$ determines a constant $m=m(X)$ such that whenever $G$ is a non-virtually cyclic group acting freely cocompactly on $X$, there exists a cubical hyperlane $h$ such that for any finite generating set $T$ for $G,$ there exist elements $g_1,g_2 \in G$ with:

\begin{enumerate}
    \item $|g_i|_T<m,$
    \item $g_1$ flips $h^-,$
    \item $g_2$ flips $h^+$, and
    \item $(g_1.g_2) h^+ \subsetneq h^+$.
\end{enumerate}

In particular, $g_1.g_2$ skewers $h$ and $g_1.g_2$ is rank-one.

\end{corollary}

\begin{proof} Up to passing to a finite-index subgroup, we may assume that $G$ is torsion-free. By Remark \ref{rmk:ABD}, since $G$ is irreducible and not a quasi-line, the maximal domains $\calC S$ is unbounded and is not a quasi-line. Let $T$ be a finite generating set for $G.$ If all elements of $T$ are loxodromic on $\calC S,$ the statement is easy and we leave it as an exercise for the reader. Suppose instead that some $a \in T$ is active on some proper $U \nest S.$ By Lemma \ref{lem:lower_bound_translation}, we can replace $U$ be a domain in $\vvbig(a)$, in particular, $\tau_U(a^{M})>A$ for a constant $A$ depending only on $\mathfrak S.$ Using Lemma \ref{lem:short_loxodromic_top}, we obtain an integer $m=m(\tau_U(a), \mathfrak S)=m(A,\mathfrak S)=m(\mathfrak S)=m(X)$ and an element $g \in G$ with $|g_i|_T<m$  and $\dist_S(\rho^{U}_S, \rho^{gU}_S)>30E.$ Up to replacing $g_i$ by a uniform power, Corollary \ref{cor:flip_cube} provides us with the desired statement.

\end{proof}

\bibliography{bio}{}
\bibliographystyle{alpha}
\end{document}